\documentclass[a4paper,12pt]{article}
\usepackage{amssymb,fullpage,amsmath,amsthm}

\newtheorem{thm}{Theorem}[section]

\newtheorem{cor}[thm]{Corollary}
\newtheorem{conj}[thm]{Conjecture}
\newtheorem{lem}[thm]{Lemma}
\newtheorem{propn}[thm]{Proposition}
\newtheorem{rem}[thm]{Remark}

\newcommand{\tv}{\tilde{V}}
\newcommand{\ty}{\tilde{Y}}
\newcommand{\bn}{\mathbb{N}}
\newcommand{\bp}{\mathbf{P}}
\newcommand{\be}{\mathbf{E}}
\newcommand{\br}{\mathbb{R}}
\newcommand{\hm}{\hat{m}}
\newcommand{\diatre}{{\rm diam}_{d_{\mathcal{T}_{ij}}}(\mathcal{T}_{ij})}

\begin{document}
\title{Spectral asymptotics for stable trees} \author{David Croydon\footnote{Dept of Statistics,
University of Warwick, Coventry, CV4 7AL, UK;
{d.a.croydon@warwick.ac.uk.}}\hspace{20pt}and\hspace{20pt}Ben
Hambly\footnote{Mathematical Institute, 24-29 St Giles', Oxford, OX1
3LB, UK; {hambly@maths.ox.ac.uk}.}\\
\tiny{\hspace{5pt}UNIVERSITY OF WARWICK\hspace{55pt}UNIVERSITY OF
OXFORD}} \maketitle

\begin{abstract}
We calculate the mean and almost-sure leading order behaviour of the high
frequency asymptotics of the eigenvalue counting function associated
with the natural Dirichlet form on $\alpha$-stable trees, which lead in
turn to short-time heat kernel asymptotics for these random structures.
In particular, the conclusions we obtain demonstrate that the spectral
dimension of an $\alpha$-stable tree is almost-surely equal to $2\alpha/(2\alpha-1)$,
matching that of certain related discrete models. We also show that the exponent
for the second term in the asymptotic expansion of the eigenvalue counting function is no greater than $1/(2\alpha-1)$. To prove our results,
we adapt a self-similar fractal argument previously applied to the continuum random tree,
replacing the decomposition of the continuum tree at the branch point of three suitably
chosen vertices with a recently developed spinal decomposition for $\alpha$-stable trees.
\end{abstract}

\section{Introduction}

This work contains a study of the spectral properties of the class of random real trees known as $\alpha$-stable trees,
$\alpha\in(1,2]$. Such objects are natural: arising as the scaling limits of conditioned Galton-Watson trees \cite{Aldous3},
\cite{Duqap}; admitting constructions in terms of Levy processes \cite{LeGallDuquesne2} and fragmentation processes \cite{hmg};
as well as having connections to continuous state branching process models \cite{LeGallDuquesne2}. In recent years, a
number of geometric properties of $\alpha$-stable trees have been studied, such as the Hausdorff dimension and measure function
\cite{LegallDuquesne}, \cite{dlegh}, \cite{hmg}, degree of branch points \cite{LegallDuquesne} and decompositions into subtrees
\cite{HPW}, \cite{Mier}, \cite{Mier1}. Here, our goal is to enhance this understanding of $\alpha$-stable trees by establishing
various analytical properties for them, including determining their spectral dimension, with the results we obtain extending
those known to hold for the continuum random tree \cite{hamcroy}, which corresponds to the case $\alpha=2$.

To allow us to state our main results, we will start by introducing some of the notation that will be used throughout the
article (precise definitions are postponed until Section \ref{stablesec}). First, fix $\alpha\in(1,2]$ and let
$\mathcal{T}=(\mathcal{T},d_\mathcal{T})$ represent the $\alpha$-stable tree $\mathcal{T}$ equipped with its natural
metric $d_\mathcal{T}$. For $\mathbf{P}$-a.e. realisation of $\mathcal{T}$, it is possible to define a canonical non-atomic
Borel probability measure, $\mu$ say, whose support is equal to $\mathcal{T}$, where $\mathbf{P}$ is the probability measure
on the probability space upon which all the random variables of the discussion are defined. As with other measured real trees,
by applying results of \cite{Kigamidendrite}, one can check that it is  possible to construct an associated Dirichlet form
on $L^2(\mathcal{T},\mu)$ as an electrical energy when we consider $(\mathcal{T},d_\mathcal{T})$ to be a resistance network,
$\mathbf{P}$-a.s. We will denote this form by $\mathcal{E}$ and its domain by $\mathcal{F}$. Our focus will be on the
asymptotic growth of the eigenvalues of the triple $(\mathcal{E},\mathcal{F},\mu)$, which are defined to be the numbers
$\lambda$ which satisfy
\[\mathcal{E}(f,g)=\lambda\int_\mathcal{T}fgd\mu,\hspace{20pt}\forall
g\in\mathcal{F},\]
for some non-trivial eigenfunction
$f\in\mathcal{F}$. The corresponding eigenvalue counting function,
$N$,  is obtained by setting
\begin{equation} \label{ecf}
N(\lambda):=\#\{\mbox{eigenvalues of
}(\mathcal{E},\mathcal{F}, \mu)\leq\lambda\}.
\end{equation}
Our conclusions for this function are presented in the following theorem, which describes the large $\lambda$ mean and
$\mathbf{P}$-a.s. behaviour of $N$. In the statement of the result, the notation $\mathbf{E}$ represents
the expectation under the probability measure $\mathbf{P}$. Note that the first order result for $\alpha=2$ was established previously
as \cite{hamcroy}, Theorem 2, and our proof is an adaptation of the argument followed there. In particular, in \cite{hamcroy}
the recursive self-similarity of the continuum random tree described in \cite{Aldous5} was used to enable renewal and branching process
techniques to be applied to deduce the results of interest. In this article, we proceed similarly by drawing recursive
self-similarity for $\alpha$-stable trees from a spinal decomposition proved in \cite{HPW}.

{\thm \label{mainthm} For each $\alpha\in(1,2]$ and $\varepsilon>0$, there exists a deterministic constant $C\in(0,\infty)$ such that the following statements hold.\\
(a) As $\lambda\rightarrow\infty$,
\[\mathbf{E}N(\lambda)= C\lambda^{\frac{\alpha}{2\alpha-1}}+O\left(\lambda^{\frac{1}{2\alpha-1}+\varepsilon}\right).\]
(b) $\mathbf{P}$-a.s., as $\lambda\rightarrow\infty$,
\[N(\lambda)\sim C\lambda^{\frac{\alpha}{2\alpha-1}}.\]
Moreover, in $\mathbf{P}$-probability, the second order estimate of part (a) also holds.}

{\rem\label{mainrem} In the special case when $\alpha=2$, the estimate of the second order term can be improved to $O(1)$
in part (a) of the above theorem. A similar comment also applies to Corollaries \ref{cor1}(a) and \ref{cor2} below.}
\bigskip

For a bounded domain $\Omega\subseteq \mathbb{R}^n$, Weyl's Theorem establishes for the Dirichlet or Neumann Laplacian eigenvalue counting function the limit
\[ \lim_{\lambda\rightarrow\infty}\frac{N(\lambda)}{\lambda^{n/2}} = c_n|\Omega|_n, \]
where $|\Omega|_n$ is the $n$-dimensional Lebesgue measure of $\Omega$ and $c_n$ is a dimension dependent constant.
As a result, in the literature on fractal sets, the limit, when it exists,
\[d_S=2\lim_{\lambda\rightarrow\infty}\frac{\ln N(\lambda)}{\ln\lambda}\]
is frequently referred to as the spectral dimension of a (Laplacian on a) set. In our setting, the previous theorem allows us to immediately  read off that an $\alpha$-stable
tree has $d_S=2\alpha/(2\alpha-1)$, $\mathbf{P}$-a.s., where the Laplacian considered here is that
associated with the Dirichlet form $(\mathcal{E},\mathcal{F})$ in the standard way.
As the Hausdorff dimension with respect to $d_\mathcal{T}$ of an
$\alpha$-stable tree $\mathcal{T}$ is $d_H=\alpha/(\alpha-1)$ (see \cite{LegallDuquesne}, \cite{hmg}), it follows that
$d_S=2d_H/(d_H+1)$, thus confirming that $\alpha$-stable trees satisfy an equality between the analytically defined $d_S$
and geometrically defined $d_H$ that has likewise been proved for various other finitely ramified random fractals when the
Hausdorff dimension is measured with respect to an intrinsic resistance metric (which is identical to $d_\mathcal{T}$ in
the $\alpha$-stable tree case), see \cite{Hamasymp}, \cite{Kigami} for example. Furthermore, it is worth remarking that
$2\alpha/(2\alpha-1)$ is also the spectral dimension of the random walk on a Galton-Watson tree whose offspring distribution
lies in the domain of attraction of a stable law with index $\alpha$, conditioned to survive \cite{CK}. This final observation
could well have been expected given the convergence result proved in \cite{DAC} that links the random walks on a related family
of Galton-Watson trees conditioned to be large and the Markov process $X$ corresponding to $(\mathcal{E},\mathcal{F},\mu)$, which
can be interpreted as the Brownian motion on the $\alpha$-stable tree.

Of course we have shown much more than just the existence of the spectral dimension, as we have demonstrated the mean and $\mathbf{P}$-a.s. existence of
the Weyl limit (which does not exist for exactly self-similar fractals with a high degree of symmetry \cite{Kigami}).
In fact, for a compact manifold with smooth boundary (under a certain geometric condition), it was proved in \cite{Ivrii} that the asymptotic expansion of the  eigenvalue counting function of the Neumann Laplacian is given by
\[ N(\lambda) = c_n |\Omega| \lambda^{n/2} + \frac{1}{4}c_{n-1} |\partial \Omega|_{n-1} \lambda^{(n-1)/2} + o(\lambda^{(n-1)/2}). \]
Analogously, the result we establish here provides an estimate on the size of the second order term for $\alpha$-stable trees. If our expansion had the same structure as the classical result, in the case $\alpha=2$, for example, we would expect to see a constant second order term, as the natural boundary is finite. However, despite seeing this in mean, we do not have (or expect) an almost sure or in probability second term of this type. Indeed, although our results do not confirm that the second order exponent is equal to $1/(2\alpha-1)$, we anticipate that the randomness in the structure leads to fluctuations of this higher order.

As in \cite{hamcroy}, it is straightforward to transfer our conclusions regarding the leading order spectral asymptotics of $\alpha$-stable trees
to a result about the heat kernel $(p_t(x,y))_{x,y\in\mathcal{T}}$ for the Laplacian associated with $(\mathcal{E},\mathcal{F},\mu)$.
In particular, a simple application of an Abelian theorem yields the following asymptotics for the trace of the heat semigroup.

{\cor \label{cor1} If $\alpha\in(1,2]$, $\varepsilon>0$, $C$ is the constant of Theorem \ref{mainthm} and $\Gamma$ is the standard
gamma function, then the following statements hold.\\
(a) As $t\rightarrow 0$,
\[\mathbf{E}\int_\mathcal{T}p_t(x,x)\mu(dx)=
C\Gamma\left(\tfrac{3\alpha-1}{2\alpha-1}\right)t^{-\frac{\alpha}{2\alpha-1}}+O\left(t^{-\frac{1}{2\alpha-1}+\varepsilon}\right).\]
(b) $\mathbf{P}$-a.s., as $t\rightarrow 0$,
\[\int_\mathcal{T}p_t(x,x)\mu(dx)\sim C\Gamma\left(\tfrac{3\alpha-1}{2\alpha-1}\right)t^{-\frac{\alpha}{2\alpha-1}}.\]}

Finally, $\alpha$-stable trees are known to satisfy the same root invariance property as the continuum random tree. More specifically,
if we select a $\mu$-random vertex $\sigma\in\mathcal{T}$, then the tree $\mathcal{T}$ rooted at $\sigma$ has the same distribution
as the tree $\mathcal{T}$ rooted at its original root, $\rho$ say (see \cite{LegallDuquesne}, Proposition 4.8). This allows us to
transfer part (a) of the previous result to a limit for the annealed on-diagonal heat kernel at $\rho$ (cf. \cite{hamcroy},
Corollary 4).

{\cor \label{cor2} If $\alpha\in(1,2]$, $\varepsilon>0$, $C$ is the constant of Theorem \ref{mainthm} and $\Gamma$ is the standard gamma function, then, as $t\rightarrow\infty$,
\[ \mathbf{E}p_t(\rho,\rho)=
C\Gamma\left(\tfrac{3\alpha-1}{2\alpha-1}\right)t^{-\frac{\alpha}{2\alpha-1}}+O\left(t^{-\frac{1}{2\alpha-1}+\varepsilon}\right).\]}

The rest of the article is organised as follows. In Section \ref{stablesec} we describe some simple properties of Dirichlet forms on
compact real trees, and also introduce a spinal decomposition for $\alpha$-stable trees that will be applied recursively. In
Section \ref{spectralsec} we prove the mean spectral result stated in this section, via a direct renewal theorem proof. By making the changes to \cite{hamcroy} that were briefly described above, we then proceed to establishing the almost-sure first order eigenvalue asymptotics in Section \ref{spectralsecas} using a branching process argument. Finally, in Section \ref{secondordersec}, we further investigate the second order behaviour of the function $N(\lambda)$ as $\lambda\rightarrow\infty$.

\section{Dirichlet forms and recursive spinal decomposition}\label{stablesec}

Before describing the particular properties of $\alpha$-stable trees that will be of interest to us, we present a brief
introduction to Dirichlet forms on more general tree-like metric spaces. To this end, for the time being we suppose that
$\mathcal{T}=(\mathcal{T},d_\mathcal{T})$ is a deterministic compact real tree (see \cite{rrt}, Definition 1.1) and $\mu$ is
a non-atomic finite Borel measure on $\mathcal{T}$ of full support. These assumptions easily allow us to check the conditions
of \cite{Kigamidendrite}, Theorem 5.4, to deduce that there exists a unique local regular Dirichlet form $(\mathcal{E},\mathcal{F})$
on $L^2(\mathcal{T},\mu)$ associated with the metric $d_\mathcal{T}$ through, for every $x,y\in\mathcal{T}$,
\begin{equation}\label{tres}
d_\mathcal{T}(x,y)^{-1}=\inf\{\mathcal{E}(f,f):\:f\in\mathcal{F},\:f(x)=0,\:f(y)=1\}.
\end{equation}
Given the triple $(\mathcal{E},\mathcal{F},\mu)$, we define the corresponding eigenvalue counting function $N$ as at (\ref{ecf}).
Now, one of the defining features of a Dirichlet form is that, equipped with the norm $\|\cdot\|_{\mathcal{E},\mu}$ defined by
\begin{equation}\label{enorm}
\|f\|_{\mathcal{E},\mu}:=\left(\mathcal{E}(f,f)+\int_{\mathcal{T}}f^2d\mu\right)^{1/2},\hspace{20pt}\forall f\in\mathcal{F},
\end{equation}
the collection of functions $\mathcal{F}$ is a Hilbert space, and moreover, the characterisation of $(\mathcal{E},\mathcal{F})$
at (\ref{tres}) implies that the natural embedding from $(\mathcal{F},\|\cdot\|_{\mathcal{E},\mu})$ into $L^2(\mathcal{T},\mu)$
is compact (see \cite{Kigres}, Lemma 8.6, for example). By standard theory for self-adjoint operators, it follows that $N(\lambda)$
is zero for $\lambda<0$ and finite for $\lambda\geq 0$ (see \cite{Kigami}, Theorem B.1.13, for example). Furthermore, by applying
results of \cite{Kigami}, Section 2.3, one can deduce that $1\in\mathcal{F}$, and $\mathcal{E}(f,f)=0$ if and only if $f$ is
constant on $\mathcal{T}$. Thus $N(0)=1$. When we incorporate this fact into our argument in the next section it will be convenient
to have notation for the shifted eigenvalue counting function $\tilde{N}:\mathbb{R}\rightarrow\mathbb{R}_+$ defined by setting
$\tilde{N}(\lambda)=N(\lambda)-1$, which clearly satisfies $\tilde{N}(\lambda)=\#\{\mbox{eigenvalues of
}(\mathcal{E},\mathcal{F}, \mu)\in(0,\lambda]\}$ for $\lambda> 0$.

Later, it will also be useful to consider the Dirichlet eigenvalues of $(\mathcal{E},\mathcal{F},\mu)$ when the boundary
of $\mathcal{T}$ is assumed to consist of two distinguished vertices $\rho,\sigma\in\mathcal{T}$, $\rho\neq \sigma$. To define
these eigenvalues precisely, we first introduce the form $(\mathcal{E}^D,\mathcal{F}^D)$ by setting
$\mathcal{E}^D:=\mathcal{E}|_{\mathcal{F}^D\times\mathcal{F}^D}$, where
$\mathcal{F}^D:=\left\{f\in\mathcal{F}:f(\rho)=0=f(\sigma)\right\}$.
Since $\mu(\{\rho,\sigma\})=0$, \cite{FOT}, Theorem 4.4.3, implies that $(\mathcal{E}^D,\mathcal{F}^D)$ is a regular Dirichlet form
on $L^2(\mathcal{T},\mu)$. Furthermore, as it is the restriction of $(\mathcal{E},\mathcal{F})$, we can apply \cite{KigLap},
Corollary 4.7, to deduce that
\begin{equation}\label{bracket}
N^D(\lambda)\leq N(\lambda)\leq N^D(\lambda)+2,
\end{equation}
where $N^D$ is the eigenvalue counting function for $(\mathcal{E}^D,\mathcal{F}^D,\mu)$, and also, since $\mathcal{E}(f,f)=0$ if
and only if $f$ is a constant on $\mathcal{T}$, $N^D(0)=0$. The eigenvalues of the triple $(\mathcal{E}^D,\mathcal{F}^D,\mu)$
will also be called the Dirichlet eigenvalues of $(\mathcal{E},\mathcal{F},\mu)$ and $N^D$ the Dirichlet eigenvalue counting
function of $(\mathcal{E},\mathcal{F},\mu)$.

To conclude this general discussion of Dirichlet forms on compact real trees, we prove a lemma that provides a lower bound for
the first non-zero eigenvalue of $(\mathcal{E},\mathcal{F},\mu)$ and first eigenvalue of $(\mathcal{E}^D,\mathcal{F}^D,\mu)$,
which will be repeatedly applied in the subsequent section. In the statement of the result, ${\rm
diam}_{d_\mathcal{T}}(\mathcal{T}):=\sup_{x,y\in\mathcal{T}}d_{\mathcal{T}}(x,y)$ is the diameter of the real tree
$(\mathcal{T},d_\mathcal{T})$.

{\lem \label{firste} In the above setting, $N^D(\lambda)=\tilde{N}(\lambda)=0$ whenever
\[0\leq \lambda<\frac{1}{{\rm diam}_{d_\mathcal{T}}(\mathcal{T})\mu(\mathcal{T})}.\]}
\begin{proof} As in the proof of \cite{hamcroy}, Lemma 20, observe that if $f\in\mathcal{F}^D$ is an eigenfunction of
$(\mathcal{E}^D,\mathcal{F}^D,\mu)$ with eigenvalue $\lambda>0$, then (\ref{tres}) implies that, for $x\in\mathcal{T}$,
\[f(x)^2=(f(x)-f(\rho))^2\leq \mathcal{E}(f,f)d_\mathcal{T}(\rho,x)\leq \lambda {\rm diam}_{d_\mathcal{T}}(\mathcal{T})
\int_\mathcal{T}f^2d\mu.\]
Integrating out $x$ with respect to $\mu$ yields the result in the Dirichlet case.

Similarly, if $f\in\mathcal{F}$ is an eigenfunction of $(\mathcal{E},\mathcal{F},\mu)$ with eigenvalue $\lambda>0$, then, for
$x,y\in\mathcal{T}$,
\[(f(x)-f(y))^2\leq \lambda {\rm diam}_{d_\mathcal{T}}(\mathcal{T}) \int_\mathcal{T}f^2d\mu.\]
Since by the definition of an eigenfunction $\int_\mathcal{T}fd\mu=\lambda^{-1}\mathcal{E}(f,1)=0$, integrating out both $x$ and $y$
with respect to $\mu$ completes the proof.
\end{proof}

We now turn to $\alpha$-stable trees. To fix notation, as in the introduction we will henceforth assume that
$\mathcal{T}=(\mathcal{T},d_\mathcal{T})$ is an $\alpha$-stable tree, $\alpha\in(1,2]$, $\mu$ is the canonical Borel probability
measure on $\mathcal{T}$ and all the random variables we consider are defined on a probability space with probability measure
$\mathbf{P}$. Since $\alpha$-stable trees have been reasonably widely studied, we do not feel it essential to provide an explicit
construction of such objects, examples of which can be found in \cite{LeGallDuquesne2} and \cite{hmg}. Instead, we simply observe
that the results of \cite{LeGallDuquesne2} imply that $(\mathcal{T},\mu)$ satisfies all the properties for measured compact real
trees that were assumed at the start of this section, and therefore the above discussion applies to the Dirichlet forms
$(\mathcal{E},\mathcal{F})$, $(\mathcal{E}^D,\mathcal{F}^D)$, and eigenvalue counting functions $N$, $\tilde{N}$, $N^D$, associated
with the $\alpha$-stable tree $\mathcal{T}$, $\mathbf{P}$-a.s.

Fundamental to our proof of Theorem \ref{mainthm} is the fine spinal decomposition of $\mathcal{T}$ that was developed in \cite{HPW},
and which we now describe. First, suppose that there is a distinguished vertex $\rho\in\mathcal{T}$, which we call the root, and
choose a second vertex $\sigma\in\mathcal{T}$ randomly according to $\mu$. Note that, since $\mu$ is non-atomic, $\rho\neq
\sigma$, $\mathbf{P}$-a.s. Secondly, let $(\mathcal{T}_i^o)_{i\in\mathbb{N}}$ be the connected components of $\mathcal{T}\backslash
[[\rho,\sigma]]$, where $[[\rho,\sigma]]$ is the minimal arc connecting $\rho$ to $\sigma$ in $\mathcal{T}$. We assume that
$(\mathcal{T}_i^o)_{i\in\mathbb{N}}$ have been ordered so that the masses $\Delta_i:=\mu(\mathcal{T}_i^o)$, which $\mathbf{P}$-a.s.
take values in $(0,1)$ and sum to 1, are non-increasing in $i$. $\mathbf{P}$-a.s. for each $i$, the closure of $\mathcal{T}_i^o$ in
$\mathcal{T}$ contains precisely one point more than $\mathcal{T}_i^o$, $\rho_i$ say, and we can therefore write it as
$\mathcal{T}_i=\mathcal{T}_i^o\cup\{\rho_i\}$. We define a metric $d_{\mathcal{T}_i}$ and probability measure $\mu_i$ on
$\mathcal{T}_i$ by setting
\[d_{\mathcal{T}_i}:=\Delta^{\frac{1-\alpha}{\alpha}}_id_{\mathcal{T}}|_{\mathcal{T}_i\times\mathcal{T}_i},\hspace{20pt}
\mu_i(\cdot):=\frac{\mu(\cdot\cap \mathcal{T}_i)}{\Delta_i}.\]
Furthermore, let $\sigma_i$ be $\mu_i$-random vertices of $\mathcal{T}_i$, chosen independently for each $i$. The usefulness of
this decomposition of $\mathcal{T}$ into the subsets $(\mathcal{T}_i)_{i\in\mathbb{N}}$ is contained in the subsequent proposition,
which is a simple modification of parts of \cite{HPW}, Corollary 10, and is stated without proof.

{\propn\label{spinal} For every $\alpha\in(1,2)$, $\left\{((\mathcal{T}_i,d_{\mathcal{T}_i}),\mu_i,\rho_i,\sigma_i)\right\}_{i\in
\mathbb{N}}$ is an independent collection of copies of $((\mathcal{T},d_{\mathcal{T}}),\mu,\rho,\sigma)$, and moreover, the entire
family is independent of $(\Delta_i)_{i\in \mathbb{N}}$, which has a Poisson-Dirichlet $(\alpha^{-1},1-\alpha^{-1})$ distribution.}
\bigskip

Similarly to the argument of \cite{hamcroy}, we will apply this result recursively, and will label the objects generated by this
procedure using the address space of sequences that we now introduce. For $n\geq 0$, let
\[\Sigma_n:=\mathbb{N}^n,\hspace{20pt}\Sigma_*:=\bigcup_{m\geq
0}\Sigma_m,\]
where $\Sigma_0:=\{\emptyset\}$. For $i\in\Sigma_m, j\in\Sigma_n$, write $ij=i_1\dots
i_m j_1 \dots j_n$, and for $k\in\Sigma_*$, denote by $|k|$ the unique integer $n$ such that $k\in\Sigma_n$. Later, we will also write for $i\in\Sigma_m$, $i|_n=i_1\dots i_n$ for any $n\leq m$.

Continuing with our inductive procedure, given $((\mathcal{T}_i,d_{\mathcal{T}_i}),\mu_i,\rho_i,\sigma_i)$ for some $i\in\Sigma_*$,
we define $\left\{((\mathcal{T}_{ij},d_{\mathcal{T}_{ij}}),\mu_{ij},\rho_{ij},\sigma_{ij})\right\}_{j\in \mathbb{N}}$ and
$(\Delta_{ij})_{j\in\mathbb{N}}$ from $((\mathcal{T}_i,d_{\mathcal{T}_i}),\mu_i,\rho_i,\sigma_i)$ using exactly the same method as
that by which $\mathcal{T}$ was decomposed above. Thus, if the  $\sigma$-algebra generated by the random variables
$(\Delta_i)_{1\leq |i|\leq n}$ is denoted  by $\mathcal{F}_n$ for each $n\in\mathbb{N}$, by iteratively applying
Proposition~\ref{spinal} it is easy to deduce the following result.

{\cor \label{recur}Let $\alpha\in(1,2)$. For each $n\in\mathbb{N}$,
$\left\{((\mathcal{T}_i,d_{\mathcal{T}_i}),\mu_i,\rho_i,\sigma_i)\right\}_{i\in \Sigma_n}$ is an independent collection of copies of
$((\mathcal{T},d_{\mathcal{T}}),\mu,\rho,\sigma)$, independent of $\mathcal{F}_n$.}
\bigskip

Finally, for $i\in\Sigma_*\backslash\{\emptyset\}$, we will write $(\mathcal{E}_i,\mathcal{F}_i)$,
$(\mathcal{E}_i^D,\mathcal{F}_i^D)$, $N_i$, $\tilde{N}_i$, $N^D_i$ to represent the Dirichlet forms and eigenvalue counting
functions corresponding to $((\mathcal{T}_i,d_{\mathcal{T}_i}),\mu_i,\rho_i,\sigma_i)$. and set
\[D_i:=\Delta_{i|_1}\Delta_{i|_2}\dots\Delta_{i|_{|i|}},\]
which is actually the mass of $\mathcal{T}_i$ with respect to the original measure $\mu$. By convention, we set
$D_\emptyset:=1$, and when other objects are indexed by $\emptyset$, we are referring to the relevant quantities defined
from the original $\alpha$-stable tree.

\section{Mean spectral asymptotics}\label{spectralsec}

To prove the mean spectral asymptotics for $\alpha$-stable trees given in Theorem \ref{mainthm}(a), we will appeal to a renewal theorem argument. In doing this, we depend on a series of inequalities that allow the Neumann and Dirichlet eigenvalue counting functions of $(\mathcal{E},\mathcal{F},\mu)$ to be usefully
compared with those associated with Dirichlet forms on subsets of $\mathcal{T}$. In particular, the collection of subsets that we
consider will be those arising from the fine spinal decomposition of $\mathcal{T}$ described in Section \ref{stablesec}, namely
$(\mathcal{T}_i)_{i\in\mathbb{N}}$, and the first main result of this section is the following, where throughout this section we
suppose $\alpha\in(1,2)$ and define $\gamma:=\alpha/(2\alpha-1)$.

{\propn \label{comparison} $\mathbf{P}$-a.s., we have, for every $\lambda \geq 0$,
\[\sum_{i\in\mathbb{N}}N_i^D(\lambda \Delta_i^{1/\gamma})\leq N^D(\lambda)\leq {N}(\lambda)\leq
1+\sum_{i\in\mathbb{N}}\tilde{N}_i(\lambda \Delta_i^{1/\gamma}),\]
with the upper bound being finite.}
\bigskip

To derive this result, we will proceed via a sequence of lemmas. The first of these provides an alternative description of
$(\mathcal{E},\mathcal{F})$ that will be useful in proving the lower bound for $N^D(\lambda)$, which appears as Lemma \ref{lower}.
We write $(\mathcal{E}_{[[\rho,\sigma]]},\mathcal{F}_{[[\rho,\sigma]]})$ to represent the local regular Dirichlet form on the
compact real tree $([[\rho,\sigma]],d_\mathcal{T}|_{[[\rho,\sigma]]\times[[\rho,\sigma]]})$ equipped with the one-dimensional
Hausdorff measure that is constructed using \cite{Kigamidendrite}, Theorem 5.4 and which therefore satisfies the variational
equality analogous to (\ref{tres}). Note that in what follows we apply the convention that if a form $E$ is defined for functions
on a set $A$ and $f$ is a function defined on $B\supseteq A$, then we write ${E}(f, f)$ to mean ${E}(f|_A, f|_A)$.

{\lem \label{edecomp} $\mathbf{P}$-a.s., we can write
\begin{equation}\label{edash}
\mathcal{E}(f,f)=\mathcal{E}_{[[\rho,\sigma]]}(f,f)+\sum_{i\in\mathbb{N}}\Delta_i^{\frac{1-\alpha}{\alpha}}\mathcal{E}_i(f,f),
\hspace{20pt}\forall f\in\mathcal{F},
\end{equation}
\begin{equation}\label{fdash}
\mathcal{F}=\left\{f\in L^2(\mathcal{T},\mu): \mbox{$f|_{[[\rho,\sigma]]}\in\mathcal{F}_{[[\rho,\sigma]]}$, and also, for
every $i\in\mathbb{N}$, $f|_{\mathcal{T}_i}\in\mathcal{F}_{i}$}\right\}.
\end{equation}}
\begin{proof} Let $(\mathcal{E}',\mathcal{F}')$ be defined by setting $\mathcal{E}'(f,f)$ to be equal to the expression on the
right-hand side of (\ref{edash}) for any $f\in\mathcal{F}'$, where $\mathcal{F}'$ is defined to be equal to the right-hand side
of (\ref{fdash}). By results of \cite{Kigami}, Section 2.3, to show that $(\mathcal{E},\mathcal{F})$ and
$(\mathcal{E}',\mathcal{F}')$ are equal and establish the lemma, it will be enough to check that (\ref{tres}) still holds when
we replace $(\mathcal{E},\mathcal{F})$ by $(\mathcal{E}',\mathcal{F}')$.

Suppose $x\in\mathcal{T}_i^o$, $y\in\mathcal{T}_j^o$, for some $i\neq j$, then the infimum of interest can be rewritten as
\begin{eqnarray*}
\lefteqn{\inf\{\mathcal{E}'(f,f):\:f\in\mathcal{F}',\:f(x)=0,\:f(y)=1\}}\\
&=&\inf_{a,b\in\mathbb{R}}\inf\{\mathcal{E}'(f,f):\:f\in\mathcal{F}',f(x)=0,f(\rho_i)=a,f(\rho_j)=b,f(y)=1\}.
\end{eqnarray*}
Now, observe that if $f$ is in the collection of functions over which this double-infimum is taken, then so is $g$, where $g$
is equal to $f$ on $[[\rho,\sigma]]\cup\mathcal{T}_i\cup\mathcal{T}_j$ and equal to $f(\rho_k)$ on $\mathcal{T}_k$ for
$k\neq i,j$. Moreover, $g$ satisfies
\[\mathcal{E}'(g,g)=\Delta_i^{\frac{1-\alpha}{\alpha}}\mathcal{E}_i(f,f)+\mathcal{E}_{[[\rho,\sigma]]}(f,f)+\Delta_j^{\frac{1-\alpha}
{\alpha}}\mathcal{E}_j(f,f)\leq\mathcal{E}'(f,f),\]
and so we can neglect functions that are not constant on each $\mathcal{T}_k$, $k\neq i,j$. In particular, we need to compute
\[\inf_{a,b\in\mathbb{R}}\inf\{\Delta_i^{\frac{1-\alpha}{\alpha}}\mathcal{E}_i(f,f)+\mathcal{E}_{[[\rho,\sigma]]}(f,f)+
\Delta_j^{\frac{1-\alpha}{\alpha}}\mathcal{E}_j(f,f)\},\]
where the second infimum is taken over functions in $\mathcal{F}'$ that satisfy $f(x)=0,f(\rho_i)=a,f(\rho_j)=b,f(y)=1$ and are
constant on each $\mathcal{T}_k$, $k\neq i,j$. Since the forms $\mathcal{E}_i$, $\mathcal{E}_{[[\rho,\sigma]]}$ and $\mathcal{E}_j$
are zero on constant functions, we can apply their characterisation in terms of distance to obtain that this is equal to
\[\inf_{a,b\in\mathbb{R}}\left\{\frac{a^2}{d_\mathcal{T}(x,\rho_i)}+\frac{(b-a)^2}{d_\mathcal{T}(\rho_i,\rho_j)}+\frac{(1-b)^2}
{d_\mathcal{T}(\rho_j,y)}\right\},\]
and, from this, a simple quadratic optimisation using the additivity of the metric $d_\mathcal{T}$ along paths yields the desired
result in this case. The argument is similar for other choices of $x,y\in\mathcal{T}$.
\end{proof}

The method of proof of the next lemma is an adaptation of \cite{KigLap}, Proposition 6.3.

{\lem \label{lower} $\mathbf{P}$-a.s., we have, for every $\lambda \geq 0$,
\[N^D(\lambda)\geq \sum_{i\in\mathbb{N}}N_i^D(\lambda \Delta_i^{1/\gamma}).\]}
\begin{proof} First, define a quadratic form $(\mathcal{E}^{(0)},\mathcal{F}^{(0)})$ by setting
$\mathcal{E}^{(0)}:=\mathcal{E}|_{\mathcal{F}^{(0)}\times \mathcal{F}^{(0)}}$, where
\[\mathcal{F}^{(0)}:=\left\{f\in \mathcal{F}:f(x)=0, \forall x\in [[\rho,\sigma]]\cup\left(
\cup_{i\in\mathbb{N}}\{\sigma_i\}\right)\right\}.\]
Since $\mu([[\rho,\sigma]]\cup\left( \cup_{i\in\mathbb{N}}\{\sigma_i\}\right))=0$, it is possible to check that
$(\mathcal{E}^{(0)},\mathcal{F}^{(0)})$ is a regular Dirichlet form on $L^2(\mathcal{T},\mu)$ by applying \cite{FOT}, Theorem 4.4.3.
Moreover, since we have that $\mathcal{F}^{(0)}\subseteq \mathcal{F}^D$ and
$\mathcal{E}^{(0)}=\mathcal{E}^D|_{\mathcal{F}^{(0)}\times \mathcal{F}^{(0)}}$, we can again apply \cite{KigLap}, Theorem 4.5, to
deduce that $N^{(0)}(\lambda)\leq N^D(\lambda)$ for every $\lambda\geq 0$, where $N^{(0)}$ is the eigenvalue counting function for
$(\mathcal{E}^{(0)},\mathcal{F}^{(0)},\mu)$. Consequently, to complete the proof of the lemma, it will suffice to show that
$\mathbf{P}$-a.s. we have, for every $\lambda \geq 0$,
\begin{equation}\label{n1bound}
N^{(0)}(\lambda)\geq \sum_{i\in\mathbb{N}}N_i^D(\lambda \Delta_i^{1/\gamma}).
\end{equation}
To demonstrate that this is indeed the case, first fix $i\in \mathbb{N}$ and suppose $f$ is an eigenfunction of
$(\mathcal{E}_i^D,\mathcal{F}_i^D,\mu_i)$ with eigenvalue $\lambda \Delta_i^{1/\gamma}$. If we set
\[g(x):=\left\{\begin{array}{ll}
          f(x), & \mbox{for }x\in \mathcal{T}_i, \\
          0 & \mbox{otherwise,}
        \end{array}\right.
\]
then we can apply Lemma \ref{edecomp} to deduce that, for $h\in\mathcal{F}^{(0)}$,
\[{\mathcal{E}}^{(0)}(g,h)=\Delta_i^{\frac{1-\alpha}{\alpha}}\mathcal{E}_i^D(f,h)=\lambda \Delta_i\int_{\mathcal{T}_i}
fhd\mu_i=\lambda \int_\mathcal{T} gh d\mu.\]
Thus $g$ is an eigenfunction of $({\mathcal{E}}^{(0)},{\mathcal{F}}^{(0)},\mu)$ with eigenvalue $\lambda$, and (\ref{n1bound})
follows.
\end{proof}

We now prove the upper bound for $N(\lambda)$. In establishing the corresponding estimates in \cite{hamcroy}, \cite{Hamasymp} and \cite{KigLap}, extensions of the Dirichlet form of interest for which the eigenvalue counting function could easily be controlled were constructed, and we will follow a similar approach here. However, since the collection of sets $(\mathcal{T}_i)_{i\in\mathbb{N}}$ is infinite, compactness issues prevent us from directly imitating this procedure to define a single suitable Dirichlet form extension of $(\mathcal{E},\mathcal{F})$. Instead we will consider a sequence of Dirichlet form extensions, each built as a sum of Dirichlet forms on the sets in a finite decomposition of $\mathcal{T}$.

{\lem \label{upperbound} $\mathbf{P}$-a.s., we have, for every $\lambda \geq 0$,
\[\tilde{N}(\lambda)\leq \sum_{i\in\mathbb{N}}\tilde{N}_i(\lambda \Delta_i^{1/\gamma}),\]
with the upper bound being finite.}
\begin{proof} We start by describing our sequence of Dirichlet form extensions of $(\mathcal{E},\mathcal{F})$. Fix $k\in\mathbb{N}$, and set $\mathcal{S}_k:=\mathcal{T}\backslash \cup_{i=1}^k\mathcal{T}_i^o$, which is a compact real tree when equipped with the restriction of $d_\mathcal{T}$ to $\mathcal{S}_k$. Again appealing to \cite{Kigamidendrite}, Theorem 5.4, let $(\mathcal{E}_{\mathcal{S}_k},\mathcal{F}_{\mathcal{S}_k})$ be the associated local regular Dirichlet form on $L^2({\mathcal{S}_k},\mu(\cdot\cap\mathcal{S}_k))$. Now, define a pair $(\mathcal{E}^{(k)},\mathcal{F}^{(k)})$ by setting $\mathcal{F}^{(k)}$ equal to
\[\left\{f\in L^2(\mathcal{T},\mu): \begin{array}{r}
                                                         \mbox{for every $i\in\{1,\dots,k\}$, $f=f_{i}$ on $\mathcal{T}_i^o$}\\
                                                         \mbox{for some $f_i\in\mathcal{F}_{i}$, and also
                                                         $f|_{\mathcal{S}_k}\in\mathcal{F}_{\mathcal{S}_k}$}
                                                   \end{array}\right\},\]
and
\[\mathcal{E}^{(k)}(f,g):=\mathcal{E}_{\mathcal{S}_k}(f,g)+\sum_{i=1}^k\Delta_i^{\frac{1-\alpha}{\alpha}}\mathcal{E}_i(f_i,g_i),\hspace{20pt}\forall f,g\in\mathcal{F}^{(k)}.\]
Since $\mathcal{F}_i$ is dense in $L^2(\mathcal{T}_i,\mu(\cdot \cap \mathcal{T}_i))$ and $\mathcal{F}_{\mathcal{S}_k}$ is dense in $L^2(\mathcal{S}_k,\mu(\cdot\cap\mathcal{S}_k))$, we clearly have that $\mathcal{F}^{(k)}$ is dense in $L^2(\mathcal{T},\mu)$. Furthermore, applying the corresponding properties for the Dirichlet forms in the sum, it is easy to check that $(\mathcal{E}^{(k)},\mathcal{F}^{(k)})$ is a non-negative symmetric bilinear form satisfying the Markov property, by which we mean that if $f \in\mathcal{F}^{(k)}$ and $\overline{f} := (0 \vee f) \wedge 1$, then $\overline{f}\in\mathcal{F}^{(k)}$ and $\mathcal{E}^{(k)}(\overline{f},\overline{f})\leq \mathcal{E}^{(k)}(f,f)$. Hence to prove that $(\mathcal{E}^{(k)},\mathcal{F}^{(k)})$ is a Dirichlet form on $L^2(\mathcal{T},\mu)$ it remains to demonstrate that $(\mathcal{F}^{(k)},\|\cdot\|_{\mathcal{E}^{(k)},\mu})$ is a Hilbert space, where $\|\cdot\|_{\mathcal{E}^{(k)},\mu}$ is the defined as at (\ref{enorm}). Given that the number of terms in the above sum is finite, this is elementary, and so $(\mathcal{E}^{(k)},\mathcal{F}^{(k)})$ is indeed a Dirichlet form on $L^2(\mathcal{T},\mu)$. Moreover, by a simple adaptation of the proof of \cite{KigLap}, Proposition 6.2(3), it can also be shown that the identity map from $(\mathcal{F}^{(k)},\|\cdot\|_{\mathcal{E}^{(k)},\mu})$ to $L^2(\mathcal{T},\mu)$ is compact, and so the eigenvalue counting function for $(\mathcal{E}^{(k)},\mathcal{F}^{(k)},\mu)$, $N^{(k)}$ say, is finite everywhere on the real line.

In order to demonstrate that $(\mathcal{E}^{(k)},\mathcal{F}^{(k)})$ is an extension of $(\mathcal{E},\mathcal{F})$, we first observe that, by following an identical line of reasoning to that applied in the proof of Lemma \ref{edecomp}, it is possible to prove that the Dirichlet form $(\mathcal{E},\mathcal{F})$ satisfies
\[\mathcal{E}(f,g):=\mathcal{E}_{\mathcal{S}_k}(f,g)+\sum_{i=1}^k\Delta_i^{\frac{1-\alpha}{\alpha}}\mathcal{E}_i(f,g),\hspace{20pt}\forall f,g\in\mathcal{F},\]
\[\mathcal{F}=\left\{f\in L^2(\mathcal{T},\mu): \begin{array}{r}
                                                         \mbox{for every $i\in\{1,\dots,k\}$, $f|_{\mathcal{T}_i}\in\mathcal{F}_{i}$, and also $f|_{\mathcal{S}_k}\in\mathcal{F}_{\mathcal{S}_k}$}
                                                   \end{array}\right\}.\]
From this characterisation of $(\mathcal{E},\mathcal{F})$, it is immediate that $\mathcal{F}\subseteq \mathcal{F}^{(k)}$ and $\mathcal{E}=\mathcal{E}^{(k)}|_{\mathcal{F}\times \mathcal{F}}$, as desired. Consequently a further application of \cite{KigLap}, Theorem 4.5, yields that $\tilde{N}(\lambda)\leq \tilde{N}^{(k)}(\lambda):=N^{(k)}(\lambda)-1$, and we complete our proof by establishing suitable upper bounds for $\tilde{N}^{(k)}$.

Let $f\not\equiv 0$ be an eigenfunction of $(\mathcal{E}^{(k)},\mathcal{F}^{(k)})$ with eigenvalue $\lambda>0$. If $i\in\{1,\dots,k\}$ and $g\in\mathcal{F}_{i}$, then define a function $h\in \mathcal{F}^{(k)}$ by setting
\[h(x):=\left\{\begin{array}{ll}
                 g(x),&\mbox{if }x\in\mathcal{T}_{i}^o,\\
                 0,&\mbox{otherwise.}
               \end{array}\right.\]
By the definition of $\mathcal{E}^{(k)}$ and this construction, we have that
\[\mathcal{E}_{i}(f,g)=\Delta^{\frac{\alpha-1}{\alpha}}_{i}\mathcal{E}^{(k)}(f,h)=\lambda\Delta^{\frac{\alpha-1}{\alpha}}_{i}\int_{\mathcal{T}}fhd\mu=\lambda\Delta^{1/\gamma}_{i}\int_{\mathcal{T}_{i}}fgd\mu_{i}.\]
Thus if $f$ is not identically zero on $\mathcal{T}_{i}$, then it must be the case that $\lambda\Delta_{i}^{1/\gamma}$ is an eigenvalue of $(\mathcal{E}_{i},\mathcal{F}_{i},\mu_{i})$. Similarly, if $g\in\mathcal{F}_{\mathcal{S}_k}$, $h$ is defined by
\[h(x):=\left\{\begin{array}{ll}
                 g(x),&\mbox{if }x\in\mathcal{S}_k,\\
                 0,&\mbox{otherwise,}
               \end{array}\right.\]
and $f$ is not identically zero on $\mathcal{S}_k$, then $\lambda$ is an eigenvalue of $(\mathcal{E}_{\mathcal{S}_k},\mathcal{F}_{\mathcal{S}_{k}},\mu(\cdot\cap\mathcal{S}_k))$. Combining these facts, it follows that, for $\lambda\geq 0$,
\begin{equation}\label{n2bound}
\tilde{N}^{(k)}(\lambda)\leq \tilde{N}_{\mathcal{S}_k}(\lambda)+\sum_{i=1}^k \tilde{N}_i(\lambda\Delta_i^{1/\gamma}),
\end{equation}
where $\tilde{N}_{\mathcal{S}_k}$ is the (strictly positive) eigenvalue counting function for $(\mathcal{E}_{\mathcal{S}_{k}},\mathcal{F}_{\mathcal{S}_{k}},\mu(\cdot\cap\mathcal{S}_k))$.

Now note that, by Lemma \ref{firste}, the first term in (\ref{n2bound}) is zero whenever $\lambda$ is strictly less than ${1}/{{\rm{diam}}_{d_\mathcal{T}}(\mathcal{S}_k)\mu(\mathcal{S}_k)}$. Thus we can conclude that, for each $k\in\mathbb{N}$,
\[\tilde{N}(\lambda)\leq\sum_{i=1}^k \tilde{N}_i(\lambda\Delta_i^{1/\gamma}),\hspace{20pt}\forall\lambda< \frac{1}{{\rm{diam}}_{d_\mathcal{T}}(\mathcal{T})(1-\Delta_1-\dots-\Delta_k)}.\]
Since ${\rm{diam}}_{d_\mathcal{T}}(\mathcal{T})<\infty$ and $\Delta_1+\dots+\Delta_k\rightarrow1$ as $k\rightarrow\infty$, $\mathbf{P}$-a.s., the upper bound of the lemma follows.

It still remains to show the $\mathbf{P}$-a.s. finiteness of $\sum_{i\in\mathbb{N}} \tilde{N}_i(\lambda\Delta_i^{1/\gamma})$. To show this is the case, we again apply Lemma \ref{firste} to obtain that the $i$th term is zero whenever \[\lambda<{\Delta_i^{-{1/\gamma}}}\left({{\rm{diam}}_{d_{\mathcal{T}_i}}(\mathcal{T}_i)\mu_i(\mathcal{T}_i)}\right)^{-1}
={}\left({\Delta_i}{\rm{diam}}_{d_{\mathcal{T}}}(\mathcal{T}_i)\right)^{-1}.\]
The result is readily obtained from this on noting that $(\Delta_i{\rm{diam}}_{d_{\mathcal{T}}}(\mathcal{T}_i))^{-1}$ is bounded below by $(\Delta_i{\rm{diam}}_{d_{\mathcal{T}}}(\mathcal{T}))^{-1}\rightarrow\infty$ as $i\rightarrow\infty$ and so only a finite number of terms (each of which is finite) in the sum are non-zero, $\mathbf{P}$-a.s.
\end{proof}

Given the eigenvalue counting function comparison result of Proposition \ref{comparison}, which follows from (\ref{bracket}), Lemma \ref{lower} and Lemma \ref{upperbound},  we now turn to our renewal theorem argument to derive mean spectral asymptotics for $\alpha$-stable trees. Similarly to \cite{hamcroy}, define the functions $(\eta_i)_{i\in\Sigma_*}$ by, for $t\in\mathbb{R}$,
\[\eta_i(t):=N_i^D(e^t)-\sum_{j\in\mathbb{N}}N_{ij}^D(e^t \Delta_{ij}^{1/\gamma}),\]
and let $\eta:=\eta_\emptyset$. By Proposition \ref{comparison}, $\eta_i(t)$ is non-negative and finite for every $t\in\mathbb{R}$, $\mathbf{P}$-a.s., and the dominated convergence theorem implies that $\eta_i$ has cadlag paths, $\mathbf{P}$-a.s. Furthermore, if we set $X_i(t):=N^D_i(e^t)$, and $X:=X_\emptyset$, then it is immediate that the following evolution equation holds:
\begin{equation}\label{evo}
X(t)=\eta(t)+\sum_{i\in\mathbb{N}}X_i(t+\gamma^{-1}\ln \Delta_i).
\end{equation}
We now introduce associated discounted mean processes
\[m(t):=e^{-\gamma t}\mathbf{E}X(t),\hspace{20pt}u(t):=e^{-\gamma t}\mathbf{E}\eta(t),\]
define a measure $\nu$ by $\nu([0,t])=\sum_{i\in\mathbb{N}} \mathbf{P}(\Delta_i\geq e^{-\gamma t})$, and let $\nu_\gamma$ be the measure that satisfies $\nu_\gamma(dt)=e^{-\gamma t}\nu(dt)$. The properties we require of $m$, $u$ and $\nu_\gamma$ are collected in the following lemma. In the proof of this result, which is an adaptation of \cite{hamcroy}, Lemma 20, it will be convenient to define, for $x\geq 0$,
\begin{equation}\label{phi}
\psi(x):=\sum_{i\in\mathbb{N}}\mathbf{E}(\Delta_i^x).
\end{equation}
By \cite{PY}, equation (6), this quantity is infinite for $x\leq \alpha^{-1}$, and otherwise satisfies
\begin{equation}\label{phiequal}
\psi(x)=\frac{\alpha-1}{\alpha x-1}.
\end{equation}
Moreover, we set
\begin{equation}\label{betadef}
\beta:=\frac{\alpha-1}{2\alpha-1}\equiv{\gamma}-\frac{1}{2\alpha-1}.
\end{equation}

\begin{lem}\label{lem:renconds}
(a) The function $m$ is bounded.\\
(b) The function $u$ is in $L^1(\mathbb{R})$ and, for any $\varepsilon>0$, $u(t)=O(e^{-(\beta-\varepsilon)t})$ as $t\rightarrow \infty$.\\
(c) The measure $\nu_\gamma$ is a Borel probability measure on $[0,\infty)$, and the integral $\int_0^\infty t\nu_\gamma(dt)$ is finite.
\end{lem}
\begin{proof} First observe that by iterating (\ref{evo}) we obtain for each $k\in\mathbb{N}$ that
\[X(t)=\sum_{|i|<k}\eta_i(t+\gamma^{-1}\ln D_i)+\sum_{i\in\Sigma_k}X_i(t+\gamma^{-1}\ln D_i).\]
Thus establishing the $\mathbf{P}$-a.s. limit
\begin{equation}\label{limit}
\lim_{k\rightarrow\infty}\sum_{i\in\Sigma_k}X_i(t+\gamma^{-1}\ln D_i)= 0,\hspace{20pt}\forall t\in\mathbb{R},
\end{equation}
will also confirm that we can $\mathbf{P}$-a.s. write
\begin{equation}\label{xexp}
X(t)=\sum_{i\in\Sigma_*}\eta_i(t+\gamma^{-1}\ln D_i),\hspace{20pt}\forall t\in\mathbb{R}.
\end{equation}
To prove that (\ref{limit}) does indeed hold, we first note that, since $X_i(t+\gamma^{-1}\ln D_i)=N^D_i(e^tD_i^{1/\gamma})=0$ for $e^tD_i^{1/\gamma}<{\rm diam}_{d_{\mathcal{T}_i}}({\mathcal{T}_i})^{-1}$, the sum appearing in (\ref{limit}) is zero if \[\sup_{i\in\Sigma_k}D_i^{1/\gamma}{\rm diam}_{d_{\mathcal{T}_i}}({\mathcal{T}_i})<e^{-t}.\]
Hence, to prove (\ref{limit}), it will be enough to show that this supremum converges $\mathbf{P}$-a.s. to zero as $k\rightarrow\infty$. To establish that this is the case, we will apply the following bound: for $\varepsilon,\theta>0$,
\begin{eqnarray}
\sum_{k=0}^\infty \mathbf{P}\left(\sup_{i\in\Sigma_k}D_i^{1/\gamma}{\rm diam}_{d_{\mathcal{T}_i}}({\mathcal{T}_i})\geq \varepsilon\right)&\leq&\sum_{k=0}^\infty \mathbf{P}\left(\sum_{i\in\Sigma_k}D_i^{\theta/\gamma}{\rm diam}_{d_{\mathcal{T}_i}}({\mathcal{T}_i})^\theta\geq \varepsilon^\theta\right)\nonumber\\
&\leq& \varepsilon^{-\theta}\mathbf{E}\left({\rm diam}_{d_{\mathcal{T}}}({\mathcal{T}})^\theta\right) \sum_{k=0}^\infty \sum_{i\in\Sigma_k}\mathbf{E}\left(D_i^{\theta/\gamma}\right)\nonumber\\
&=&\varepsilon^{-\theta}\mathbf{E}\left({\rm diam}_{d_{\mathcal{T}}}({\mathcal{T}})^\theta\right) \sum_{k=0}^\infty \psi(\theta\gamma^{-1})^k,\label{upperbound1}
\end{eqnarray}
where we have made use of the recursive decomposition result of Corollary \ref{recur}. Exploiting the fragmentation process description of $\alpha$-stable trees proved in \cite{Mier}, it is possible to apply \cite{Haas}, Proposition 14, to check that the expectation $\mathbf{E}({\rm diam}_{d_{\mathcal{T}}}({\mathcal{T}})^\theta)$ is finite for any $\theta>0$. Furthermore, by (\ref{phiequal}), we have that $\psi(\theta\gamma^{-1})<1$ for $\theta>\gamma$. Thus, by choosing $\theta>\gamma$, we obtain that the expression at (\ref{upperbound1}) is finite, and therefore the Borel-Cantelli lemma can be applied to complete the proof that (\ref{limit}) and (\ref{xexp}) hold.

From the characterisation of $X$ at (\ref{xexp}) and the definition of $\eta_i$ we see that
\[m(t)=e^{-\gamma t}\sum_{i\in\Sigma_*}\mathbf{E}\left(N_i^D(e^tD_i^{1/\gamma})-\sum_{j\in\mathbb{N}}N_{ij}^D(e^tD_{ij}^{1/\gamma})\right).\]
Since
\begin{eqnarray}
\eta_i(t+\gamma^{-1} \ln D_i) &=& N_i^D(e^tD_i^{1/\gamma})-\sum_{j\in\mathbb{N}}N_{ij}^D(e^tD_{ij}^{1/\gamma})\nonumber\\
&\leq& \mathbf{1}_{\{D_i^{1/\gamma}{\rm diam}_{d_{\mathcal{T}_i}}({\mathcal{T}_i})\geq e^{-t}\}}+\sum_{j\in\mathbb{N}}\left(\tilde{N}_{ij}(e^tD_{ij}^{1/\gamma})-N_{ij}^D(e^tD_{ij}^{1/\gamma})\right),\nonumber\\
&\leq& \mathbf{1}_{\{D_i^{1/\gamma}{\rm diam}_{d_{\mathcal{T}_i}}({\mathcal{T}_i})\geq e^{-t}\}}+\sum_{j\in\mathbb{N}}\mathbf{1}_{\{D_{ij}^{1/\gamma}{\rm diam}_{d_{\mathcal{T}_{ij}}}({\mathcal{T}_{ij}})\geq e^{-t}\}},\label{etabound}
\end{eqnarray}
where we have applied (\ref{bracket}), Lemma \ref{firste} and Proposition \ref{comparison}, it follows that
\begin{eqnarray}
m(t)&\leq&2e^{-\gamma t}\sum_{i\in\Sigma_*}\mathbf{P}\left(D_i^{1/\gamma}{\rm diam}_{d_{\mathcal{T}_i}}({\mathcal{T}_i})\geq e^{-t}\right)\nonumber\\
&=&2e^{-\gamma t}\mathbf{E}\left(\#\left\{i\in\Sigma_*:-\gamma^{-1} \ln D_i \leq t+\ln{\rm diam}_{d_{\tilde{\mathcal{T}}}}({\tilde{\mathcal{T}}})\right\}\right),\nonumber
\end{eqnarray}
where $(\tilde{\mathcal{T}},d_{\tilde{\mathcal{T}}})$ is an independent copy of $({\mathcal{T}},d_{{\mathcal{T}}})$. Similarly to the corresponding argument in \cite{hamcroy}, by considering the Crump-Mode-Jagers branching process with particles $i\in\Sigma_*$, where $i\in\Sigma_*$ has offspring $ij$ at time $-\ln \Delta_{ij}$ after its birth, $j\in\mathbb{N}$, it is possible to show that $\mathbf{E}(\#\{i\in\Sigma_*:\:-\ln D_i\leq t\})\leq Ce^{t}$ for every $t\in\mathbb{R}$, where $C$ is a finite constant. Hence $m(t)\leq 2C\mathbf{E}\left({\rm diam}_{d_{\mathcal{T}}}({\mathcal{T}})^{\gamma}\right)$ for every $t\in\mathbb{R}$. As already noted, the moments of the diameter of an $\alpha$-stable tree are finite and so this bound establishes that $m$ is bounded.

For part (b), first observe that
\[u(t)=e^{-\gamma t}\mathbf{E}\eta(t)\leq e^{-\gamma t}\sum_{i\in\Sigma_*}\mathbf{E}\eta_i(t+\gamma^{-1}\ln D_i)=m(t),\]
and so $u$ is bounded. Thus, since $\eta$ is $\mathbf{P}$-a.s. cadlag, then $u$ is also measurable. Furthermore, multiplying (\ref{etabound}) by $e^{-\gamma t}$ and taking expectations yields, for any $\theta>0$,
\begin{eqnarray*}
u(t)&\leq & e^{-\gamma t}\left(\mathbf{P}\left({\rm diam}_{d_{\mathcal{T}}}({\mathcal{T}})\geq e^{-t}\right)+\sum_{i\in\mathbb{N}}\mathbf{P}\left(\Delta_{i}^{1/\gamma}{\rm diam}_{d_{\mathcal{T}_{i}}}({\mathcal{T}_{i}})\geq e^{-t}\right)\right)\\
&\leq &e^{(\theta-\gamma)t}\mathbf{E}\left({\rm diam}_{d_{\mathcal{T}}}({\mathcal{T}})^\theta\right)\left(1+\sum_{i\in\mathbb{N}}\mathbf{E}\left(\Delta_i^{\theta/\gamma}\right)\right)\\
&=& C_{\theta} e^{(\theta-\gamma)t},
\end{eqnarray*}
where the second inequality is a simple application of Chebyshev's inequality and $C_{\theta}:=\mathbf{E}({\rm diam}_{d_{\mathcal{T}}}({\mathcal{T}})^\theta)(1+\psi(\theta\gamma^{-1}))$. As all the positive moments of ${\rm diam}_{d_{\mathcal{T}}}({\mathcal{T}})$ are finite and $\psi(\theta\gamma^{-1})$ is finite for $\theta>\gamma\alpha^{-1}$, $C_\theta$ is a finite constant for any $\theta>(2\alpha-1)^{-1}$. In particular, choosing $\theta=(2\alpha-1)^{-1}+\varepsilon$, we obtain $u(t)=O(e^{-(\beta-\varepsilon)t})$ as $t\rightarrow \infty$, which is the second claim of part (b). We further note that by setting $\theta=1+\gamma$, the above bound implies $u(t)=O(e^{t})$ as  $t\rightarrow -\infty$, which, in combination with our earlier observations, establishes that $u\in L^1(\mathbb{R})$ as desired.

Finally, to demonstrate that $\nu_\gamma$ is a Borel probability measure on $[0,\infty)$ is elementary given that $\psi(1)=\sum_{i\in \mathbb{N}}\Delta_i=1$, $\mathbf{P}$-a.s. Moreover, by definition the integrability condition can be rewritten $\sum_{i\in\mathbb{N}}\mathbf{E}(\Delta_i|\ln\Delta_i|)<\infty$, and this can be confirmed by a second application of equation (6) of \cite{PY}.
\end{proof}

Applying this lemma, it would be possible to apply the renewal theorem of \cite{Karlin} exactly as in \cite{hamcroy} to deduce the convergence of $m(t)$ as $t\rightarrow\infty$. However, in order to establish an estimate for the second order term, we present a direct proof of the renewal theorem in our setting. The $\beta$ in the statement of the result is defined as at (\ref{betadef}), and $m(\infty)$ is the constant defined by
\begin{equation}\label{minf}
m(\infty):=\frac{\int_{-\infty}^\infty
u(t)dt}{\int_0^\infty t\nu_\gamma(dt)}.
\end{equation}
That $m(\infty)$ is finite and non-zero is an easy consequence of Lemma \ref{lem:renconds}.

{\propn \label{propn:convrate} For any $\varepsilon>0$, the function $m$ satisfies
\[\left|m(t)-m(\infty)\right|=O(e^{-(\beta-\varepsilon)t}),\]
as $t\rightarrow \infty$.}
\begin{proof} From (\ref{xexp}) and Fubini's theorem, we obtain
\begin{eqnarray*}
m(t)&=&e^{-\gamma t}\mathbf{E}X(t)\\
&=&\sum_{i\in\Sigma_*} e^{-\gamma t}\mathbf{E}\eta_i(t+\gamma^{-1}\ln D_i)\\
&=&\sum_{i\in\Sigma_*} \int_0^\infty e^{-\gamma(t-s)}\mathbf{E}\eta_i(t-s)e^{-\gamma s}\mathbf{P}(-\gamma^{-1}\ln D_i\in ds)\\
&=& \int_0^\infty u(t-s)\sum_{i\in\Sigma_*}e^{-\gamma s}\mathbf{P}(-\gamma^{-1}\ln D_i\in ds).
\end{eqnarray*}
We will analyse the measure in this integral. Let $\lambda>0$, then
\begin{eqnarray*}
\int_0^\infty e^{-\lambda s} \sum_{i\in\Sigma_*}e^{-\gamma s}\mathbf{P}(-\gamma^{-1}\ln D_i\in ds) &=& \sum_{i\in\Sigma_*} \mathbf{E}D_i^{1+\lambda/\gamma}\\
&=&\sum_{n=0}^\infty \psi(1+\lambda\gamma^{-1})^n\\
&=& \frac{1}{1-\psi(1+\lambda\gamma^{-1})}.
\end{eqnarray*}
Furthermore, observe that $M:=(\int_0^\infty s\nu_\gamma(ds))^{-1}$ satisfies
\[M^{-1}=-\gamma^{-1}\psi'(1)=\frac{2\alpha-1}{\alpha-1}.\]
It follows that
\[\int_0^\infty e^{-\lambda s}\left[Mds-\sum_{i\in\Sigma_*}e^{-\gamma s}\mathbf{P}(-\gamma^{-1}\ln D_i\in ds)\right]=\frac{M}{\lambda}- \frac{1}{1-\psi(1+\lambda\gamma^{-1})}=-1,\]
and inverting this Laplace transform yields
\[Mds-\sum_{i\in\Sigma_*}e^{-\gamma s}\mathbf{P}(-\gamma^{-1}\ln D_i\in ds)=-\delta_0(s)ds,\]
where $\delta_0(s)$ is the Dirac delta function. Therefore
\begin{eqnarray*}
\lefteqn{m(\infty)-m(t)}\\
&=& M\int_0^\infty u(t+s)ds + \int_0^\infty u(t-s)\left[Mds-\sum_{i\in\Sigma_*}e^{-\gamma s}\mathbf{P}(-\gamma^{-1}\ln D_i\in ds)\right]\\
&=& M\int_0^\infty u(t+s)ds - u(t),
\end{eqnarray*}
and the result follows from Lemma \ref{lem:renconds}.
\end{proof}

Rewriting the above result in terms of $N^D$ and using (\ref{bracket}) to compare $N^D$ with $N$ yields Theorem \ref{mainthm}(a) for $\alpha\in(1,2)$. Before we conclude this section, though, let us briefly discuss the case $\alpha=2$, so as to explain how the corresponding parts of the theorem and Remark \ref{mainrem} can be verified. Letting $m$, $u$ and $\nu_\gamma$ be defined as in \cite{hamcroy} (which closely matches the notation of this article), then by repeating an almost identical argument to the previous proof, with the Poisson-Dirichlet random variables $(\Delta_i)_{i\in\mathbb{N}}$ of this article being replaced by the Dirichlet $(\frac{1}{2},\frac{1}{2},\frac{1}{2})$ triple of that, it is possible to show that
\[{m(\infty)-m(t)}= \int_0^\infty u(t+s)ds - u(t).\]
(To do this, it is necessary to apply the observations that, in the $\alpha=2$ setting, the constant $M$ is equal to 1, and the function corresponding to $\psi(x)$ can be computed to be $3(2x+1)^{-1}$.) Since $u$ was shown in \cite{hamcroy} to satisfy $u(t)\leq Ce^{-2t/3}$ for $t\geq 0$, the right-hand side is bounded by a constant when multiplied by $e^{2t/3}$, and it follows that Theorem \ref{mainthm}(a) holds for $\alpha=2$ with the second order term reduced to $O(1)$.

\section{Almost-sure spectral asymptotics}\label{spectralsecas}

Our task for this section is to establish the $\mathbf{P}$-a.s. convergence of $e^{-\gamma t}X(t)$ as $t\rightarrow\infty$,
where $X(t)$ is defined as in the previous section and $\alpha\in(1,2)$ is fixed throughout. For this, we follow the branching
process argument of \cite{hamcroy}, which extends \cite{Hamasymp}, making changes where necessary to deal with the infinite
number of offspring. This approach relies on a second moment bound for $X(t)$, which we prove via a sequence of lemmas. For
brevity, we will henceforth write $\delta_i:= {\rm diam}_{d_{\mathcal{T}_i}}({\mathcal{T}_i})$.  It will also be convenient to let
$m(i,j) = \sup\{n:i|_n=j|_n\}$ be the generation of the most recent
common ancestor of the addresses $i,j \in \Sigma_*$ and for $j=ik$ to write
$D_j^i = \prod_{l=|i|+1}^{|j|} \Delta_{j|_l}$.

We first state an elementary extension of Markov's inequality.

\begin{lem}
Let $X,Y$ be positive random variables. Then for all $x,y>0$,
\begin{equation}
 \bp(X>x,Y>y) \leq \frac{1}{xy}\be XY. \label{eq:mi2}
\end{equation}
\end{lem}

\begin{lem}\label{lem:pijest}
For $i\in\Sigma_k$, $j\in\Sigma_l$ with $k\leq l$ and $\theta>0$, we have that
\begin{eqnarray*}
\lefteqn{\bp(D_i^{1/\gamma} \delta_i \geq e^{-t},D_j^{1/\gamma} \delta_j \geq e^{-t}) }\\
& \leq&  e^{2\theta t} (\be \delta^{4\theta})^{1/2}
(\be(\Delta_{i|_{m+1}}^{4\theta/\gamma})\be(\Delta_{j|_{m+1}}^{4\theta/\gamma}))^{1/4}
\be( D_{i|_m}^{2\theta/\gamma}) \be\left((D_i^{i|_{m+1}})^{\theta/\gamma} \right) \be\left((D_j^{j|_{m+1}})^{\theta/\gamma}\right)
\end{eqnarray*}
whenever $m<k$, and if $\varepsilon>0$, then
\[\bp(D_i^{1/\gamma} \delta_i \geq e^{-t},D_j^{1/\gamma} \delta_j \geq e^{-t})  \leq e^{2\theta t} \be (\delta^{2(1+\varepsilon^{-1})\theta})^{1/(1+\varepsilon^{-1})}\be (D_{i}^{2\theta/\gamma})  \be((D_j^i)^{(1+\varepsilon)\theta/\gamma})^{1/(1+\varepsilon)}\]
whenever $m=k$, where $m:=m(i,j)$ and $\delta:=\delta_\emptyset$.
\end{lem}
\begin{proof} We start by assuming $m<k$ or, if $k=l$, then $m<k-1$.
By definition, we have that
\begin{eqnarray}
\lefteqn{\bp(D_i^{1/\gamma} \delta_i \geq e^{-t},D_j^{1/\gamma}\delta_j \geq e^{-t})}\nonumber\\
&=& \bp(D_{i|_{m+1}}^{1/\gamma} (D_i^{i|_{m+1}})^{1/\gamma}\delta_i \geq
e^{-t},D_{j|_{m+1}}^{1/\gamma} (D_j^{j|_{m+1}})^{1/\gamma}\delta_j\geq e^{-t}) \nonumber \\
&=& \be(\bp((D_i^{i|_{m+1}})^{1/\gamma} \geq x_i, (D_j^{j|_{m+1}})^{1/\gamma}\geq x_j|x_i,x_j)) \label{eq:stage1}
\end{eqnarray}
where $x_i^{-1} = e^t D_{i|_{m+1}}^{1/\gamma} \delta_i,x_j^{-1} = e^t D_{j|_{m+1}}^{1/\gamma}
\delta_j$. Now, as $D_i^{i|_{m+1}}$ and $D_j^{j|_{m+1}}$ are independent, we have
\begin{eqnarray*}
&&\bp(D_i^{1/\gamma}\delta_i \geq
e^{-t},D_j^{1/\gamma}\delta_j \geq e^{-t})\\
&& \qquad\leq \be(\bp((D_i^{i|_{m+1}})^{1/\gamma} \geq x_i|x_i,x_j)\bp( (D_j^{j|_{m+1}})^{1/\gamma} \geq x_j|x_i,x_j)) \\
&& \qquad \leq \be(x_i^{-\theta}x_j^{-\theta} \be((D_i^{i|_{m+1}})^{\theta/\gamma} |x_i,x_j)\be(
(D_j^{j|_{m+1}})^{\theta/\gamma}|x_i,x_j)) \\
&& \qquad \leq e^{2\theta t} \be(\delta_i^\theta\delta_j^{\theta}\Delta_{i|_{m+1}}^{\theta/\gamma}\Delta_{j|_{m+1}}^{\theta/\gamma})
\be(D_{i|_m}^{2\theta/\gamma}) \be((D_i^{i|_{m+1}})^{\theta/\gamma})\be((D_j^{j|_{m+1}})^{\theta/\gamma})
\end{eqnarray*}
A repeated application of Cauchy-Schwarz to $\be(\delta_i^{\theta}\delta_j^{\theta}\Delta_{i|_{m+1}}^{\theta/\gamma}
\Delta_{j|_{m+1}}^{\theta/\gamma})$ then gives the result.

For the case where $k=l$ and $m=k-1$, that is $i,j$ have the same parent we cannot use independence in
the same way and instead use (\ref{eq:mi2}) in (\ref{eq:stage1}) to get
\begin{eqnarray*}
\bp(D_i^{1/\gamma}\delta_i \geq
e^{-t},D_j^{1/\gamma}\delta_j \geq e^{-t})
&\leq& \be(x_i^{-\theta}x_j^{-\theta} \be(\Delta_{i|_k}^{\theta/\gamma} \Delta_{j|_k}^{\theta/\gamma}|x_i,x_j)) \\
&\leq& e^{2\theta t} \be(\delta_i^\theta\delta_j^{\theta}\Delta_{i|_k}^{\theta/\gamma}\Delta_{j|_k}^{\theta/\gamma})
\be(D_{i|_m}^{2\theta/\gamma})
\end{eqnarray*}
and Cauchy-Schwarz again gives the result.

For the case where $m=k$ we have by (\ref{eq:mi2}) that
\begin{eqnarray*}
&&\bp(D_i^{1/\gamma} \delta_i \geq e^{-t},D_i^{1/\gamma} (D_j^i)^{1/\gamma}\delta_j \geq e^{-t})\\
&& \qquad\leq e^{2\theta t} \be\left( D_i^{2\theta/\gamma}\delta_i^\theta(D^i_j)^{\theta/\gamma}\delta_j^\theta\right)\\
&& \qquad=e^{2\theta t} \be(D_i^{2\theta/\gamma})\be\left(\delta_i^\theta(D^i_j)^{\theta/\gamma}\delta_j^{\theta}\right)
\end{eqnarray*}
Applying H\"{o}lder twice to $\be\left(\delta_i^\theta(D^i_j)^{\theta/\gamma}\delta_j^{\theta}\right)$, we have the result in this case as well.
\end{proof}

For the following result, we define $\psi_r:=\psi(r\theta\gamma^{-1})$ for $r=1,2$, where the function $(\psi(x))_{x\geq 0}$ was introduced at (\ref{phi}). We also set
\[\psi_{1,\varepsilon}:=\sum_{i\in\mathbb{N}}\be(\Delta_i^{(1+\varepsilon)\theta/\gamma})^{1/(1+\varepsilon)}.\]
If $\theta>\gamma/\alpha$, we observe that $\psi_1\leq \psi_{1,\varepsilon}<\infty$, where the lower inequality is simply Jensen's and the upper inequality is a consequence of \cite{PY}, equation (50).

\begin{lem}\label{etaijest} For $k\leq l$, $\theta>\gamma/\alpha$ and $\varepsilon>0$, we have that
\[\sum_{i\in\Sigma_k} \sum_{j\in\Sigma_l} \be(\eta_i(t+\gamma^{-1}\ln D_i)\eta_j(t+\gamma^{-1}\ln D_j)) \leq C e^{2\theta t}
(k+1)\psi_{1,\varepsilon}^{k+l}\left(\frac{\psi_2}{\psi_{1,\varepsilon}^2}\vee 1\right)^k\]
for some finite constant $C$.
\end{lem}
\begin{proof} Let $i\in\Sigma_k,j\in\Sigma_l$ for some $k\leq l$, then (\ref{etabound}) implies that
\begin{eqnarray}
 \lefteqn{\be(\eta_i(t+\gamma^{-1}\ln D_i)\eta_j(t+\gamma^{-1}\ln D_j))}\nonumber\\
  &\leq& \be(\mathbf{1}_{A_i,A_j} +
\mathbf{1}_{A_i} \sum_{n\in\bn} \mathbf{1}_{A_{jn}} + \mathbf{1}_{A_j} \sum_{n\in\bn} \mathbf{1}_{A_{in}}+
\sum_{n,n'\in\bn} \mathbf{1}_{A_{in},A_{jn'}} ) \nonumber\\
&=& \bp(A_i,A_j) + \sum_{n\in\bn} (\bp(A_i,A_{jn}) + \bp(A_j,A_{in})) + \sum_{n,n'\in\bn} \bp(A_{in},A_{jn'}),\label{aibound}
\end{eqnarray}
where $A_i:=\{D^{1/\gamma}_i\delta_i\geq e^{-t}\}$. We now apply Lemma~\ref{lem:pijest} to deduce that
\begin{eqnarray*}
\lefteqn{\sum_{i\in\Sigma_k} \sum_{j\in\Sigma_l:j|_k\neq i }\be(\eta_i(t+\gamma^{-1}\ln D_i)\eta_j(t+\gamma^{-1}\ln D_j))}\\
 &\leq& e^{2\theta t} (\be \delta^{4\theta})^{1/2}\sum_{i\in\Sigma_k} \sum_{j\in\Sigma_l:j|_k\neq i}
(\be(\Delta_{i|_{m+1}}^{4\theta/\gamma})\be(\Delta_{j|_{m+1}}^{4\theta/\gamma}))^{1/4}
\be( D_{i|_m}^{2\theta/\gamma}) (I_1+I_2+I_3),
\end{eqnarray*}
where $m:=m(i,j)$ is strictly less than $k$ for the $i$ and $j$ in the above sum, $\delta:=\delta_\emptyset$, and
\begin{eqnarray*}
I_1 &:=&  \be\left((D_i^{i|_{m+1}})^{\theta/\gamma}\right)\be\left((D_j^{j|_{m+1}})^{\theta/\gamma}\right)\\
I_2 &:=& \sum_{n\in\bn} \left(
\be\left((D_i^{i|_{m+1}})^{\theta/\gamma}\right)\be\left((D_{jn}^{j|_{m+1}})^{\theta/\gamma}\right) +
\be\left((D_{in}^{i|_{m+1}})^{\theta/\gamma}\right)\be\left((D_j^{j|_{m+1}})^{\theta/\gamma}\right)\right)
\\
I_3 &:=& \sum_{n,n'\in\bn} \be\left((D_{in}^{i|_{m+1}})^{\theta/\gamma}\right)\be\left((D_{jn'}^{j|_{m+1}})^{\theta/\gamma}\right).
\end{eqnarray*}
Noting as in the proof of Lemma \ref{lem:renconds} that $\be \delta^{4\theta}$ is finite, it will suffice to bound the sums over the terms involving $I_1$, $I_2$ and $I_3$. Firstly, we have that
\begin{eqnarray*}
\lefteqn{\sum_{i\in\Sigma_k} \sum_{j\in\Sigma_l:j|_k\neq i }(\be(\Delta_{i|_{m+1}}^{4\theta/\gamma})\be(\Delta_{j|_{m+1}}^{4\theta/\gamma}))^{1/4}
\be( D_{i|_m}^{2\theta/\gamma})I_1}\\
&\leq & \sum_{m'=0}^{k-1} \sum_{i'\in\Sigma_{m'}} \be(D_{i'}^{2\theta/\gamma}) \\ &&\times \sum_{i\in\Sigma_{k}:i|_{m'}=i'}\sum_{j\in\Sigma_{l}:j|_{m'}=i'}(\be(\Delta_{i|_{m'+1}}^{4\theta/\gamma})\be(\Delta_{j|_{m'+1}}^{4\theta/\gamma}))^{1/4}
\be\left((D_i^{i|_{m'+1}})^{\theta/\gamma}\right)\be\left((D_j^{j|_{m'+1}})^{\theta/\gamma}\right)\\
&\leq& C \sum_{m'=0}^{k-1} \psi_2^{m'}\psi_1^{k+l-2m'-2}\\
&\leq & Ck \psi_1^{k+l}\left(\frac{\psi_2}{\psi_1^2}\vee 1\right)^k ,
\end{eqnarray*}
where $C$ is a finite constant and we have applied \cite{PY}, equation (50) to deal with the $(m+1)$st generation terms. Similar calculations show that the analogous sums involving $I_2$ and $I_3$ can be bounded by the same expression after suitable modification of the constant.

We now consider the sum of $\be(\eta_i(t+\gamma^{-1}\ln D_i)\eta_j(t+\gamma^{-1}\ln D_j))$ over $i\in\Sigma_k$ and $j\in\Sigma_l$ in the case when $i$ is an ancestor of $j$. Again applying Lemma \ref{lem:pijest}, we deduce that
\begin{eqnarray}
\lefteqn{\sum_{i\in\Sigma_k} \sum_{j\in\Sigma_l:j|_k= i }\left(\bp(A_i,A_j) + \bp(A_{j|_{k+1}},A_{j})+ \sum_{n\in\bn}\left( \bp(A_i,A_{jn}) + \bp(A_{jn|_{k+1}},A_{jn})\right)\right)}\nonumber\\
 &\leq& e^{2\theta t}
 \be (\delta^{2(1+\varepsilon^{-1})\theta})^{1/(1+\varepsilon^{-1})}
 \sum_{i\in\Sigma_k}  \be (D_{i}^{2\theta/\gamma})\sum_{j\in\Sigma_l:j|_k= i} \left[\vphantom{\sum_{n\in\mathbb{N}} }
  \be((D_j^i)^{(1+\varepsilon)\theta/\gamma})^{1/(1+\varepsilon)}\right.\nonumber\\
 &&+ \be (\Delta_{j|_{k+1}}^{2\theta/\gamma})\be((D_j^{j|_{k+1}})^{(1+\varepsilon)\theta/\gamma})^{1/(1+\varepsilon)}
+\sum_{n\in\mathbb{N}} \be((D_{jn}^i)^{(1+\varepsilon)\theta/\gamma})^{1/(1+\varepsilon)}\nonumber\\
&&+ \left. \be(\Delta_{jn|_{k+1}}^{2\theta/\gamma})\sum_{n\in\mathbb{N}}\be((D_{jn}^{jn|_{k+1}})^{(1+\varepsilon)\theta/\gamma})^{1/(1+\varepsilon)}
\right]\nonumber\\
&\leq & Ce^{2\theta t} \psi_2^k \psi_{1,\varepsilon}^{l-k}.\label{b1}
\end{eqnarray}
Note that if $l=k$, then the first term involving $j|_{k+1}$ should be deleted from the above argument. Another appeal to Lemma \ref{lem:pijest} yields that we also have
\begin{equation}\label{b2}
\sum_{i\in\Sigma_k}\: \sum_{j\in\Sigma_l:j|_k= i }\: \sum_{n\in\mathbb{N}:in\neq j|_{k+1}}\left(\bp(A_{in},A_j)+\sum_{n'\in\mathbb{N}}\bp(A_{in},A_{jn'})\right)
\leq Ce^{2\theta t} \psi_2^k\psi_1^{l-k}.
\end{equation}
Summing (\ref{b1}) and (\ref{b2}), the bound at (\ref{aibound}) implies
\begin{equation}\label{b3}
\sum_{i\in\Sigma_k} \sum_{j\in\Sigma_l:j|_k=i }\be(\eta_i(t+\gamma^{-1}\ln D_i)\eta_j(t+\gamma^{-1}\ln D_j))  \leq Ce^{2\theta t} \psi_2^k \psi_{1,\varepsilon}^{l-k}.
\end{equation}
On combining our estimates, we obtain the lemma.
\end{proof}

We can now proceed with our second moment bound for $X(t)$.

\begin{lem}\label{lem:2ndmom} For $\theta>\gamma$, there is a finite constant $C$ such that
\[ \mathbf{E}(X(t)^2) \leq C e^{2\theta t}, \hspace{20pt}\forall t \in \mathbb{R}. \]
\end{lem}

\begin{proof}
This is a simple application of the preceding lemma. Firstly, applying (\ref{xexp}), we have that
\begin{eqnarray*}
{\be X(t)^2}&=& \be \left(\sum_{i,j\in \Sigma_*} \eta_i(t+\gamma^{-1}\ln D_i)\eta_j(t+\gamma^{-1}\ln D_j)\right),\\
&\leq& 2\sum_{k=0}^{\infty} \sum_{l=k}^{\infty}
\sum_{i\in\Sigma_k} \sum_{j\in\Sigma_l} \be\left( \eta_i(t+\gamma^{-1}\ln D_i)\eta_j(t+\gamma^{-1}\ln D_j)\right).
\end{eqnarray*}
From Lemma \ref{etaijest}, it follows that
\[ \be \left(X(t)^2\right) \leq C e^{2\theta t}\sum_{k=0}^{\infty} \sum_{l=k}^{\infty}
(k+1)\psi_{1,\varepsilon}^{k+l}\left(\frac{\psi_2}{\psi_{1,\varepsilon}^2}\vee 1\right)^k,
\]
for some finite constant $C$, which may depend on $\varepsilon>0$. Noting that, in the range of $\theta$ considered, $\psi_r<1$ for $r=1,2$ and $\psi_{1,\varepsilon}\rightarrow\psi_1$ as $\varepsilon\rightarrow 0$ (by the dominated convergence theorem), it is clear that the double sum is finite for suitably small $\varepsilon$. Thus the proof is complete.
\end{proof}

For the purposes of proving almost-sure convergence, we introduce the
following notation to represent a cut-set of $\Sigma_*$: for $t>0$,
\[\Lambda_t:=\{i\in\Sigma_*:\:-\gamma^{-1}\ln D_i\geq t > -\gamma^{-1}\ln D_{i|_{|i|-1}}\}.\] We
will also have cause to refer to the subset of $\Lambda_t$ defined by, for $t,c>0$,
\[\Lambda_{t,c}:=\{i\in\Sigma_*:\:-\gamma^{-1}\ln D_i\geq t+c,\: t> -\gamma^{-1}\ln
D_{i|_{|i|-1}}\}.\]
We note that the sets $\Lambda_t, \Lambda_{t,c}$ are countably infinite, but that $\Lambda_t\backslash \Lambda_{t,c}$
is a finite set $\bp$-a.s. The following is the main result of this section.

{\propn $\mathbf{P}$-a.s we have that
\[ e^{-\gamma t} X(t) \to m(\infty), \;\; \text{as } t\to\infty,\]
where $m(\infty)$ is the constant defined at (\ref{minf}).}
\begin{proof}
We follow the earlier proofs of such results which originate with \cite{Nerman}.  First, we
truncate the characteristics $\eta_i$ (this term is meant in the generalised sense of \cite{Nerman}, Section 7)
by defining, for fixed $c>0$,
$\eta^c_i(t):=\eta_i(t)\mathbf{1}_{\{t\leq n_0c\}}$, where $n_0$ is an integer that
will be chosen later in the proof. From these truncated characteristics construct the
processes $X_i^c$ as
\[X_i^c(t):=\sum_{j\in\Sigma_*}\eta_{ij}^c(t+\gamma^{-1}\ln (D_{ij}/D_i)),\]
and set $X^c:=X_\emptyset^c$. The corresponding discounted mean process is $m^c(t):=e^{-\gamma t}\mathbf{E}X^c(t)$,
and this may be checked to converge to $m^c(\infty)\in(0,\infty)$ as
$t\rightarrow\infty$ using the renewal theorem of \cite{Karlin}. From a branching process decomposition of $X^c$, we can deduce the following bound for $n_1\geq n_0$, $n\in\mathbb{N}$,
\[|e^{-\gamma c(n+n_1)}X^c(c(n+n_1))-m^c(\infty)|\leq S_1(n,n_1)+S_2(n,n_1)+S_3(n,n_1),\]
where,
\begin{eqnarray*}
\lefteqn{S_1(n,n_1):=}\\
&&\left|\sum_{i\in\Lambda_{cn}\backslash\Lambda_{cn,cn_1}} \left(e^{-\gamma c (n+n_1)} X_i^c(c(n+n_1)+\gamma^{-1}\ln D_i)-D_i m^c(c(n+n_1)+\gamma^{-1}\ln D_i)\right)\right|,
\end{eqnarray*}
\[S_2(n,n_1):=\left|\sum_{i\in\Lambda_{cn}\backslash\Lambda_{cn,cn_1}} D_i m^c(c(n+n_1)+\gamma^{-1}\ln D_i)-m^c(\infty)\right|,\]
\[S_3(n,n_1):=e^{- \gamma c(n+n_1) }\sum_{i\in\Lambda_{cn,cn_1}}X_i^c(c(n+n_1)+\gamma^{-1}\ln D_i).\]
For the first two terms we can
apply exactly the same argument as in \cite{Hamasymp} to deduce that, $\mathbf{P}$-a.s.,
\[\lim_{n_1\rightarrow\infty}\limsup_{n\rightarrow\infty} S_j(n,n_1)=0,\hspace{20pt}\mbox{for }j=1,2.\]
We will now show that $S_3(n,n_1)$ decays in a similar fashion.  We need to modify the approach of \cite{hamcroy}
slightly to deal with the infinite number of offspring. Firstly we introduce a set of characteristics, $\phi_i^{c,n_1}$, defined by
\[\phi_i^{c,n_1}(t):=\sum_{j \in \bn}X_{ij}(0)
\mathbf{1}_{\{t+cn_1+\ln\delta_{ij}\geq -\gamma^{-1}\ln \Delta_{ij}>t+cn_1, \:t>0\}},\]
where the bound involving $\delta_{ij}=\diatre$ is included to ensure that only a finite number of terms contribute to the sum. For $t>0$, set
\[Y^{c,n_1}(t):=\sum_{i\in\Sigma_*}\phi_i^{c,n_1}(t+\gamma^{-1}\ln D_i).\]
Note that from the definition of the cut-set $\Lambda_{cn,cn_1}$ we can deduce that
\[Y^{c,n_1}(cn)=\sum_{i\in\Lambda_{cn,cn_1}}X_i(0)\mathbf{1}_{\{c(n+n_1)+\gamma^{-1}\ln D_i \geq -\ln\delta_i\}}\geq e^{\gamma c (n+n_1)}
S_3(n,n_1),\]
where for the second inequality we apply the monotonicity of the $X_i$s and the fact that $X_i(t)=0$ for $t<\ln\delta_i^{-1}$. Now, $Y^{c,n_1}$ is a branching process with random characteristic $\phi^{c,n_1}_i$,
and we will proceed by checking that the conditions of the extension of \cite{Nerman}, Theorem 5.4, that is stated as \cite{Hamasymp}, Theorem 3.2,
are satisfied by it. There are two conditions, one on the characteristic, the other on the reproduction process.

For the reproduction process, it is enough
to show that there is a non-increasing, bounded positive integrable function $g$ such that $\int_0^{\infty} g(t)^{-1} \nu_{\gamma}(dt)<\infty$.
If we take $g(t)=1\wedge t^{-2}$, then by equation (6) of \cite{PY}, we see that
\[ \int_0^{\infty} (1\vee t^2) e^{-\gamma t} \nu(dt) \leq \mathbf{E}\sum_{i\in\bn} \Delta_i \left(1+ (\gamma^{-1} \ln \Delta_i)^2\right) < \infty. \]

For the characteristic, we need to prove the existence of a non-increasing, bounded positive integrable function $h$ such that $\mathbf{E} \sup_{t\geq0} e^{-\gamma t}\phi^{c,n_1}_\emptyset(t)/h(t) <\infty$.  Taking $h(t):=e^{-\beta t/2}$, where $\beta$ is the constant defined at (\ref{betadef}), we find that
\begin{eqnarray}
\sup_{t\in\mathbb{R}}\frac{e^{-\gamma t}\phi_{\emptyset}^{c,n_1}(t)}{h(t)}&\leq&e^{\left(\frac{\beta}{2}-\gamma\right)t}\sum_{i\in\mathbb{N}}X_i(0)\left(e^{t+cn_1}\delta_i\Delta_i^{1/\gamma}\right)^{\frac{1+\alpha}{2(2\alpha-1)}}\nonumber\\
&= & e^{cn_1(1+\alpha)/2(2\alpha-1)} \sum_{i\in\mathbb{N}}X_i(0)\delta_i^{\frac{1+\alpha}{2(2\alpha-1)}}\Delta_i^{\frac{1+\alpha}{2\alpha}}.\label{c1exp}
\end{eqnarray}
Thus it will suffice to prove that the final expression here has a finite first moment. Since $(X_i(0))_{i\in\mathbb{N}}$ and $(\delta_i)_{i\in\mathbb{N}}$ are independent of $(\Delta_i)_{i\in\mathbb{N}}$, we deduce that
\begin{equation}\label{c1exp2}
\mathbf{E}\left(\sum_{i\in\mathbb{N}}X_i(0)\delta_i^{\frac{1+\alpha}{2(2\alpha-1)}}\Delta_i^{\frac{1+\alpha}{2\alpha}}
\right)\leq \left(\mathbf{E}(X(0)^2)\mathbf{E}\left(\delta_\emptyset^{\frac{1+\alpha}{2\alpha-1}}\right)\right)^{1/2}\psi((1+\alpha)/2\alpha),
\end{equation}
where we have applied Cauchy-Schwarz to separate the expectations involving $\delta_\emptyset$ and $X(0)$. Now observe that, by Lemma \ref{lem:2ndmom}, $\mathbf{E}(X(0)^2)< \infty$,  the moments of the diameter of a $\alpha$-stable tree are finite and $\psi((1+\alpha)/2\alpha)<\infty$, which means that the condition on the characteristics is fulfilled.

Consequently, applying \cite{Hamasymp}, Theorem 3.2, we find that $\mathbf{P}$-a.s.,
\[e^{-\gamma t}Y^{c,n_1}(t)\rightarrow \frac{\int_0^\infty e^{-\gamma t}
 \mathbf{E}\phi^{c,n_1}_\emptyset(t)dt}{\int_0^{\infty}t\nu_\gamma(dt)}, \hspace{20pt}\mbox{as }t\rightarrow\infty.\]
By (\ref{c1exp}) and (\ref{c1exp2}), the above limit is bounded by $Ce^{cn_1(1+\alpha)/2(2\alpha-1)}$, where $C$ is a constant not depending on $n_1$. Hence, $\mathbf{P}$-a.s.,
\[\lim_{n_1\rightarrow\infty}\limsup_{n\rightarrow\infty}S_3(n,n_1)\leq \lim_{n_1\rightarrow\infty} Ce^{cn_1(1+\alpha)/2(2\alpha-1)}e^{-\gamma cn_1}=\lim_{n_1\rightarrow\infty} Ce^{- cn_1\beta/2}=0,\]
and combining the three limit results for $S_1$, $S_2$ and $S_3$, it is easy to deduce that $\mathbf{P}$-a.s.,
\begin{equation}
\lim_{n\rightarrow\infty}|e^{-\gamma cn}X^c(cn)-m^c(\infty)|=0. \label{eq:cnconv}
\end{equation}

We now show that the process $X$, when suitably scaled, converges along the subsequence $(cn)_{n\geq 0}$.
From (\ref{eq:cnconv}) we have that $\mathbf{P}$-a.s.,
\begin{equation}
\limsup_{n\rightarrow\infty}|e^{-\gamma cn}X(cn)-m(\infty)|\leq |m(\infty)-m^c(\infty)|+\limsup_{n\rightarrow\infty}
e^{-\gamma cn}|X(cn)-X^c(cn)|.\label{upper}
\end{equation}
Recall that the process $X^c$ and its discounted mean process $m^c$ depend on the integer $n_0$. By the dominated
convergence theorem, the first of the terms in (\ref{upper}), which is deterministic, converges to zero as
$n_0\rightarrow \infty$. To show the corresponding result for the second term, we start by introducing a collection of random variables $(U_i)_{i\in\Sigma_*}$ satisfying
\[U_i:=\sup_{t\in\mathbb{R}}\frac{e^{-\gamma t}\eta_i(t)}{h(t)},\]
where, similarly to above, $h(t):= e^{-\beta t/2}$. By applying ideas from the proof of Lemma \ref{lem:renconds}, it is an elementary exercise to check that $\mathbf{E}U_i<\infty$. Now, if we define characteristics $\phi_i(t):=U_i\mathbf{1}_{\{t\in[0,c]\}}$, then this finite integrability of $U_i$ readily implies the conditions of \cite{Hamasymp}, Theorem 3.2, which yields that, $\mathbf{P}$-a.s.,
\[e^{-\gamma t}\sum_{i\in\Sigma_*}\phi_i(t+\gamma^{-1} \ln D_i)\rightarrow \frac{\int_0^c e^{-\gamma t}
 \mathbf{E}U_idt}{\int_0^{\infty}t\nu_\gamma(dt)}, \hspace{20pt}\mbox{as }t\rightarrow\infty.\]
This we can rewrite as, $\mathbf{P}$-a.s.,
\[e^{-\gamma t}\sum_{i\in A_{t}\backslash A_{t-c}}U_i\rightarrow \frac{\int_0^c e^{-\gamma t}
 \mathbf{E}U_idt}{\int_0^{\infty}t\nu_\gamma(dt)}, \hspace{20pt}\mbox{as }t\rightarrow\infty,\]
where $A_t:=\{i\in\Sigma_*:\:-\gamma^{-1}\ln D_i \leq t\}$. Hence, we can proceed similarly to the proof of \cite{Nerman}, Lemma 5.8, to obtain that, $\mathbf{P}$-a.s., for $n> n_0$,
\begin{eqnarray*}
e^{-\gamma cn}|X(cn)-X^c(cn)|&=&e^{- \gamma cn}\sum_{i\in\Sigma_*}\eta_i(cn+\gamma^{-1}\ln D_i)\mathbf{1}_{\{cn+\gamma^{-1}\ln D_i > cn_0 \}}\\
&\leq & \sum_{i\in\Sigma_*} D_i U_i h(cn+\gamma^{-1} \ln D_i) \mathbf{1}_{\{i\in A_{c(n-n_0)}\}}\\
& \leq & U_{\emptyset}h(cn)+ \sum_{k=1}^{n-n_0}\sum_{i\in A_{ck}\backslash A_{c(k-1)}} D_i U_i h(c(n-k))\\
& \leq & U_{\emptyset}h(cn)+ \sum_{k=1}^{n-n_0} e^{-c((n-k)\beta/2+(k-1)\gamma)}
\sum_{i\in A_{ck}\backslash A_{c(k-1)}}  U_i\\
&\leq & U_{\emptyset} e^{-\beta c n/2}+ C \sum_{k=1}^{n-n_0} e^{-c(n-k)\beta/2}\\
&= & U_{\emptyset} e^{-\beta c n/2}+ C \sum_{k=n_0}^{\infty} e^{-ck\beta/2}.
\end{eqnarray*}
This yields in particular that, $\mathbf{P}$-a.s.,
\[\limsup_{n\rightarrow\infty}e^{-\gamma cn}|X(cn)-X^c(cn)|\leq C e^{-cn_0\beta/2}.\]
Consequently, by choosing $n_0$ suitably large, the upper bound in (\ref{upper}) can be made arbitrarily small, which has as a result that $e^{-\gamma cn}X(cn)\rightarrow m(\infty)$ as $n\rightarrow \infty$, $\mathbf{P}$-a.s., for each $c$. The proposition is readily deduced from this using the monotonicity of $X$.
\end{proof}

\section{The second order term}\label{secondordersec}

In this section we proceed to extend the result of the previous section so as to obtain an estimate on the second order term. We continue to assume that $\alpha\in(1,2)$, and recall from (\ref{betadef}) the definition of $\beta=(\alpha-1)/(2\alpha-1)$. In particular, in terms of the process $X(t) = N^D(e^t)$, it is our aim to prove the following proposition.

\begin{propn}\label{lem:2ndterm}
For each $\varepsilon>0$, in $\mathbf{P}$-probability, as $t\rightarrow\infty$,
\[ |e^{-\gamma t} X(t) - m(\infty)| = O(e^{-(\beta-\varepsilon)t}).\]
\end{propn}
\bigskip

Let us start by introducing the notation $Y(t) := e^{-\gamma t} X(t) - m(t)$ for the rescaled and centred version of $X(t)$. Using the decomposition of $X$ given at (\ref{evo}), we have
\[ Y(t) = \zeta(t) + \sum_{i\in\mathbb{N}}
 \Delta_i Y_i(t+\gamma^{-1}\ln \Delta_i), \]
where
\[ \zeta(t)=e^{-\gamma t} (\eta(t)- \be\eta(t)) +
\sum_{i\in\mathbb{N}}(\Delta_im(t+\gamma^{-1}\ln \Delta_i)-\be(\Delta_im(t+\gamma^{-1}\ln \Delta_i))).\]
Hence
\begin{equation}
  Y(t)^2 =  Z(t) + \sum_{i\in\mathbb{N}} \Delta_i^2 Y_i(t+\gamma^{-1}\ln\Delta_i)^2, \label{eq:y2decomp}
\end{equation}
where
\[ Z(t) = \zeta^2(t) + 2\zeta(t)\sum_{i\in\mathbb{N}}\Delta_iY_i(t+\gamma^{-1}\ln\Delta_i) + \sum_{i,j\in\mathbb{N},i\neq j}
\Delta_i\Delta_j Y_i(t+\gamma^{-1}\ln\Delta_i)Y_j(t+\gamma^{-1}\ln\Delta_j). \]
Iterating (\ref{eq:y2decomp}), we have for any $k\in\bn$,
\[ Y(t)^2 = \sum_{|i|<k} D_i^2 Z_i(t+\gamma^{-1}\ln D_i) + \sum_{i\in\Sigma_k} D_i^2 Y_i(t+\gamma^{-1}\ln D_i)^2. \]
The following lemma shows that the value of the remainder term here converges to zero as $k\rightarrow \infty$, from which we obtain a useful decomposition of $Y(t)^2$.

\begin{lem}
We have, $\mathbf{P}$-a.s., that
\[ \lim_{k\to\infty}  \sum_{i\in\Sigma_k} D_i^2 Y_i(t+\gamma^{-1}\ln D_i)^2 = 0,\]
and hence we have the representation, $\mathbf{P}$-a.s.,
\begin{equation}
Y(t)^2 = \sum_{i\in\Sigma_*} D_i^2 Z_i(t+\gamma^{-1}\ln D_i),\hspace{20pt}\forall t\in\mathbb{R}.\label{eq:y2rep}
 \end{equation}
\end{lem}

\begin{proof}
From the second moment estimates of Lemma \ref{lem:2ndmom} and the boundedness of $m$ (see Lemma \ref{lem:renconds}), we have that
\[\be Y(t)^2 \leq 2\be e^{-2\gamma t} X(t)^2+2m(t)^2 \leq C\left(e^{2\varepsilon t}\vee 1\right).\]
Thus
\begin{eqnarray*}
 \be \sum_{i\in\Sigma_k} D_i^2 Y_i(t+\gamma^{-1}\ln D_i)^2 &\leq& \sum_{i\in\Sigma_k} C \be\left(D_i^2 \left(e^{2\varepsilon(t+\gamma^{-1}\ln D_i)}\vee 1\right)\right). \\
 &=&\sum_{i\in\Sigma_k} C \left(e^{2\varepsilon t} \vee 1\right)\be(D_i^{2}) \\
 &=& C \left(e^{2\varepsilon t} \vee 1\right) \psi(2)^k,
\end{eqnarray*}
where $\psi$ was defined at (\ref{phi}). As $\psi(2)<1$, we therefore have that
\[\sum_{k=0}^\infty\mathbf{P}\left(\sum_{i\in\Sigma_k} D_i^2 Y_i(t+\gamma^{-1}\ln D_i)^2 >\delta\right)\leq C\delta^{-1} \left(e^{2\varepsilon t} \vee 1\right) \sum_{k=0}^\infty   \psi(2)^k<\infty,\]
where we have applied Chebyshev to deduce the first inequality. Hence, Borel-Cantelli implies the representation of $Y(t)^2$ for each fixed $t$, $\mathbf{P}$-a.s. By countability, it follows that the same result holds for each rational $t$. Since $Y$ is cadlag, the representation can easily be extended to hold for all $t\in\mathbb{R}$.
\end{proof}

We use this result to derive a second moment estimate for $Y(t)$.

\begin{lem}\label{ylem} For each $\varepsilon>0$, there exists a constant $C$ such that
\[ \be Y(t)^2 \leq C e^{-(2\beta-\varepsilon)t}.\]
\end{lem}

\begin{proof}
We need to estimate the terms on the right-hand side of (\ref{eq:y2rep}). Firstly, from the definition of $Z_i$,
conditioning on $\Delta_i$ and using $\be Y_i(t) =0$ we have
\begin{eqnarray*}
\lefteqn{\be D_i^2 Z_i(t+\gamma^{-1}\ln D_i)} \\
&=& \be D_i^2 \zeta_i(t+\gamma^{-1}\ln D_i)^2 + 2\be D_i^2 \zeta_i(t+\gamma^{-1}\ln D_i)\sum_{j\in\bn} \Delta_{ij} Y_{ij}(t+\gamma^{-1}\ln D_{ij}) \\
&\leq& 2\be \left(e^{-2\gamma t}D_i^2 (\eta_i(t+\gamma^{-1}\ln D_i)- \be(\eta_i(t+\gamma^{-1}\ln
 D_i)|D_i))^2\right)+2\be\left(\left(\sum_{j\in\mathbb{N}}\kappa_{ij}\right)^2\right)\\
& & \qquad + 2\be D_i^2 e^{-\gamma t} \eta_i(t+\gamma^{-1}\ln D_i)\sum_{j\in\bn} \Delta_{ij} Y_{ij}(t+\gamma^{-1}\ln D_{ij}),
\end{eqnarray*}
where we define $\kappa_{ij}:=D_{ij}m(t+\gamma^{-1}\ln D_{ij})-
\be(D_{ij}m(t+\gamma^{-1}\ln D_{ij})|D_i)$. Hence $\be Y(t)^2 \leq 2(I_1 + I_2 + I_3)$, where
\begin{eqnarray*}
 I_1 &=& \sum_{i\in\Sigma_*}\be \left(e^{-2\gamma t} D_i^2(\eta_i(t+\gamma^{-1}\ln D_i)- \be(\eta_i(t+\gamma^{-1}\ln D_i)|D_i))^2\right), \\
 I_2 &=& \sum_{i\in\Sigma_*}\be\left(\left(\sum_{j\in\mathbb{N}}\kappa_{ij}\right)^2\right), \\
I_3 &=& \sum_{i\in\Sigma_*}\be \left(D_i^2 e^{-\gamma t} \eta_i(t+\gamma^{-1}\ln D_i)\sum_{j\in\bn} \Delta_{ij} Y_{ij}(t+\gamma^{-1}\ln D_{ij}) \right).
 \end{eqnarray*}
For $I_1$, we apply Lemma \ref{lem:pijest} similarly to the proof of Lemma \ref{etaijest} to deduce that, for suitably chosen
$\theta>\gamma/\alpha$,
\begin{eqnarray*}
 I_1 &\leq& \sum_{k=0}^\infty \sum_{i\in\Sigma_k} e^{-2\gamma t} \be(\eta_i(t+\gamma^{-1}\ln D_i)^2) \\
 &\leq & C e^{-2\gamma t} e^{2\theta t} \sum_{k=0}^\infty \psi(2\theta\gamma^{-1})^k \\
 &= & C e^{-(2(\alpha-1)\gamma/\alpha-\varepsilon)t},
\end{eqnarray*}
which is a bound of the appropriate magnitude. For $I_2$, we use an extension of \cite{PY}, equation (6), coupled with the
estimate on the convergence rate of $m(t)$ to its limit. Specifically, we begin by writing $\kappa_{ij} = D_i A_j(t_i)$
where $t_i := t+\gamma^{-1}\ln D_i$ and
\[ A_j(t) := \Delta_j m(t+\gamma^{-1}\ln \Delta_j) - \be\Delta_j m(t+\gamma^{-1}\ln \Delta_j). \]
As $\sum_{j\in\bn} \Delta_j =1$, we can write
\[ \sum_{j\in\bn} A_j(t) = \sum_{j\in\bn} \left(\Delta_j \hm(t+\gamma^{-1}\ln \Delta_j) - \be\Delta_j
\hm(t+\gamma^{-1}\ln \Delta_j)\right),\]
where $\hm(t) := m(t)-m(\infty)$. By Proposition \ref{propn:convrate} and the boundedness of $m$ (Lemma \ref{lem:renconds}(a)),
there is a constant $C$ such that $|\hm(t)| = |m(t)-m(\infty)| \leq C e^{-(\beta-\varepsilon)t}$ for $t\in\br$, and hence
\[ \left| \sum_{j\in\bn} A_j(t))\right| \leq C \left( \sum_{j\in\bn} \Delta_j^{1-(\beta-\varepsilon)/\gamma} + \be\sum_{j\in\bn} \Delta_j^{1-(\beta-\varepsilon)/\gamma}\right) e^{-(\beta-\varepsilon)t} . \]
Now, to obtain our estimate, we note that
\begin{eqnarray*}
\lefteqn{\be\left(\left(\sum_{j\in\mathbb{N}}\kappa_{ij}\right)^2\right)}\\
 &= &  \be\left(D_i^2
    \be \left(\left(\sum_{j\in\mathbb{N}}A_j(t_i)\right)^2\vline D_i\right)\right) \\
    &\leq & C \be\left(D_i^2 \be \left( \left(\sum_{j\in\mathbb{N}} \Delta_j^{1-(\beta-\varepsilon)/\gamma}
    \right)^2 + \left(\be\sum_{j\in\mathbb{N}} \Delta_j^{1-(\beta-\varepsilon)/\gamma}\right)^2\right) D_i^{-2(\beta-\varepsilon)/\gamma} \right)e^{-2(\beta-\varepsilon)t}.
\end{eqnarray*}
Using the lemma in the appendix, and the fact that $1-\beta\gamma^{-1}=\alpha^{-1}$, we can compute the first term as follows:
\begin{eqnarray*}
 \be \left(\sum_{j\in\mathbb{N}} \Delta_j^{1-(\beta-\varepsilon)/\gamma} \right)^2 &=& \be\left( \sum_{j\in\mathbb{N}} \Delta_j^{2-2(\beta-\varepsilon)/\gamma} +
\sum_{j,l\in\mathbb{N}:j\neq l} \Delta_j^{1-(\beta-\varepsilon)/\gamma}\Delta_l^{1-(\beta-\varepsilon)/\gamma} \right)\\
&= & \psi(2-2(\beta-\varepsilon)\gamma^{-1}) + \frac{\Gamma(2-\alpha^{-1})^2}{(\varepsilon\gamma^{-1})^2\Gamma(1-\alpha^{-1})\Gamma(1+\alpha^{-1})}.
\end{eqnarray*}
Thus we obtain that
\[ \be\left(\left(\sum_{j\in\mathbb{N}}\kappa_{ij}\right)^2\right) \leq C \be D_i^{2/\alpha+\epsilon}e^{-2(\beta-\varepsilon)t}. \]
As $2/\alpha + \epsilon>1$ for $\alpha \in (1,2]$, this can be summed over $i\in\Sigma_*$ to give the bound
\[ I_2 \leq c e^{-2(\beta-\varepsilon)t}. \]
Finally, for $I_3$, we first observe that by (\ref{xexp}) and the definition of $Y(t)$ we can write
\[ Y(t) = e^{-\gamma t} \sum_{i\in \Sigma_*} \left(\eta_i(t+\gamma^{-1}\ln D_i)-\be\eta_i(t+\gamma^{-1}\ln D_i)\right). \]
Hence
\begin{eqnarray*}
 I_3 &\leq& e^{-2\gamma t}\sum_{i\in\Sigma_*}  \be \left(D_i^2 \eta_i(t+\gamma^{-1}\ln D_i)\sum_{j\in \Sigma_*\backslash\{\emptyset\}}\eta_{ij}(t+\gamma^{-1}\ln D_{ij})\right) \\
&\leq &e^{-2\gamma t} \sum_{k=0}^\infty\:\sum_{l=k+1}^\infty\:\sum_{i\in\Sigma_k}\:\sum_{j\in \Sigma_l:j|_k=i}\be \left( \eta_i(t+\gamma^{-1}\ln D_i)\eta_{j}(t+\gamma^{-1}\ln D_{j})\right).
\end{eqnarray*}
To bound the inner two sums, we follow the arguments of Section \ref{spectralsecas}, but taking different powers to those used there. For example, in the case when $i$ is an ancestor of $j$, we can replace the second statement of Lemma \ref{lem:pijest} by: for $\theta_1, \theta_2,\varepsilon'\geq 0$,
\[\bp(A_i\cap A_j)\leq Ce^{t(\theta_1+\theta_2)}\be\left(D_i^{(\theta_1+\theta_2)/\gamma}\right)\be\left((D^i_j)^{\theta_2(1+\varepsilon')/\gamma}\right)^{1/(1+\varepsilon')},\]
where $A_i$ is defined as in the proof of Lemma \ref{etaijest} and $C$ is a constant that depends only on $\theta_1, \theta_2$ and $\varepsilon'$. After proceeding similarly with the other relevant terms and taking $\theta_1:=(2\alpha^{-1}-1)\gamma$, $\theta_2:=\gamma+\varepsilon$, we are consequently able to show that (cf. (\ref{b3})): for any $\theta>\gamma/\alpha$,
\[\sum_{i\in\Sigma_k} \sum_{j\in\Sigma_l:j|_k=i }\be(\eta_i(t+\gamma^{-1}\ln D_i)\eta_j(t+\gamma^{-1}\ln D_j))  \leq Ce^{2\theta t} \psi_2^k \psi(1+\varepsilon\gamma^{-1},\varepsilon')^{l-k},\]
where, as previously, $\psi_2:=\psi(2\theta\gamma^{-1})$, and \[\psi(1+\varepsilon\gamma^{-1},\varepsilon'):=\sum_{i\in\bn}\be\left(\Delta_i^{(1+\varepsilon\gamma^{-1})(1+\varepsilon')}\right)^{1/(1+\varepsilon')}
\rightarrow \psi({1+\varepsilon\gamma^{-1}}),\]
as $\varepsilon'\rightarrow 0$. Since $\psi(1+\varepsilon\gamma^{-1}),\psi_2<1$, if $\varepsilon'$ is chosen small enough, we find from these results that
\[I_3\leq C e^{-(2\beta-\varepsilon)t} \sum_{k=0}^\infty\:\sum_{l=k+1}^\infty \psi_2^k \psi(1+\varepsilon\gamma^{-1},\varepsilon')^{l-k} \leq C e^{-(2\beta-\varepsilon)t},\]
as desired.
\end{proof}

Given this bound, it is now straightforward to prove the result of interest.

\begin{proof}[Proof of Proposition~\ref{lem:2ndterm}]
By Chebyshev and Lemma \ref{ylem}, there exists a $C$ such that for all $t\geq0$
\[ \bp(|Y(t)|>x) \leq x^{-2}\be(Y(t)^2) \leq x^{-2}C e^{-(2\beta -\varepsilon) t}. \]
Now choose $x = e^{-t(\beta-\varepsilon)}$ to see that
\[ \bp(|Y(t)|>e^{-t(\beta-\varepsilon)}) \leq Ce^{-\varepsilon t},\]
and hence we have the desired result in probability.
\end{proof}

To completely establish Theorem \ref{mainthm}, it remains to demonstrate that part (b) holds in the case $\alpha=2$.
However, since the appropriate first order asymptotic behaviour was already obtained in \cite{hamcroy} and the second
order term requires us to make only very minor changes to the above argument, we omit the proof of this part of the theorem.

Unfortunately, the arguments of this section are not enough to yield an almost-sure result regarding the size of second
order term in the asymptotic expansion of the eigenvalue counting function for $\alpha$-stable trees. By Borel-Cantelli,
the results we have proved so far would be good enough to show that for any $c>0$ it is $\bp$-a.s. the case that
\[ \limsup_{n\to\infty}|Y(nc)| e^{cn(\beta-\varepsilon)} \leq 1. \]
To extend this to all $t$ and establish that, $\bp$-a.s.,
\[ \limsup_{t\to\infty}|e^{-\gamma t}X(t)-m(t)| e^{-t(\beta-\varepsilon)} \leq C,\]
it would be enough to have moment estimates of the form
\[ \be Y(t)^k \leq C e^{-k(\beta-\varepsilon)t}, \]
for all $k\in \bn$. Although it appears that suitable extensions of the techniques used here would, after much effort,
yield such a result, we will leave such a calculation to an interested reader.
Finally, let us remark that, by analogy with the results known to hold for related branching processes, it might also be hoped that a
central limit theorem-type result of the following form holds, establishing the second order term for the eigenvalue counting
function of $\alpha$-stable trees.

\begin{conj} As $\lambda\rightarrow\infty$,
\[  \frac{N^D(\lambda)-m(\infty)\lambda^{\alpha/(2\alpha-1)}}{\lambda^{1/(2\alpha-1)}} \to Z_{\alpha},\;\;\text{in distribution,} \]
 where $Z_\alpha$ is an $\alpha$-stable random variable.
\end{conj}

\appendix

\section{Appendix}

The following result, which is a straightforward extension of \cite{PY}, equation (6), is applied in the proof of Lemma \ref{ylem}.

\begin{lem} Suppose $(V_i)_{i\in\mathbb{N}}$ has the Poisson-Dirichlet $(\alpha,\theta)$ distribution. For measurable functions $f,g$ we have
\begin{eqnarray*}
\lefteqn{\be \sum_{i=1}^{\infty} \sum_{j=1,j\neq i}^{\infty} f(V_i)g(V_j)}\\
& =& C_{\alpha,\theta}
\int_0^1 \int_0^1 f(x) g((1-x)y) x^{-1-\alpha}
 (1-x)^{\theta+\alpha-1}y^{-1-\alpha}(1-y)^{\theta+2\alpha-1} dxdy,
 \end{eqnarray*}
where
\[ C_{\alpha,\theta} = \frac{\Gamma(\theta+1)\Gamma(\theta+\alpha+1) }{\Gamma(1-\alpha)^2\Gamma(\theta+\alpha)
\Gamma(\theta+2\alpha)}. \]
\end{lem}

\begin{proof} This is an application of size-biased sampling. Following the set up in \cite{PY}, define $\tv_1$ to be a size biased pick
from $(V_i)_{i\in\mathbb{N}}$, that is
\[ P(\tv_1=V_n|\{V_i\}) = V_n, \;\;n\in\bn. \]
Also, let $\tv_2$ be the second size biased pick, that is a random variable with distribution
\[ P(\tv_2=V_n|\tv_1,\{V_i\}) = \frac{V_n\mathbf{1}_{\{V_n\neq\tv_1\}}}{1-\tv_1}, \;\;n\in\bn. \]
It is then possible to show that we can write
\[ \tv_1 = \ty_1,\;\; \tv_2=(1-\ty_1)\ty_2, \]
where $\ty_i$, $i=1,2$, are independent random variables with Beta($1-\alpha,\theta+i\alpha$) distribution (see \cite{PY}, Proposition 2, for example).
Applying this result,
\begin{eqnarray*}
&& \be \sum_{i=1}^{\infty} \sum_{j=1,j\neq i}^{\infty} f(V_i)g(V_j) \\
&& \qquad = \be\frac{f(\tv_1)}{\tv_1}
\frac{g(\tv_2)(1-\tv_1)}{\tv_2} \\
&& \qquad = \be\frac{f(\ty_1)}{\ty_1} \frac{g((1-\ty_1)\ty_2)}{\ty_2} \\
&& \qquad = \frac{\Gamma(\theta+1)\Gamma(\theta+\alpha+1) }{\Gamma(1-\alpha)^2\Gamma(\theta+\alpha)\Gamma(\theta+2\alpha)}\\
&&\qquad \qquad \times\int_0^1\int_0^1 f(x)g((1-x)y) x^{-1-\alpha}y^{-1-\alpha}(1-x)^{\theta+\alpha-1}(1-y)^{\theta+2\alpha-1}dxdy,
\end{eqnarray*}
as required.
\end{proof}

To apply this in our setting, we use $f(x) = g(x)=x^{\alpha^{-1}+\varepsilon\gamma^{-1}}$ with  Poisson-Dirichlet parameters $(\alpha^{-1},1-\alpha^{-1})$.
\def\cprime{$'$}
\providecommand{\bysame}{\leavevmode\hbox to3em{\hrulefill}\thinspace}
\providecommand{\MR}{\relax\ifhmode\unskip\space\fi MR }
\providecommand{\MRhref}[2]{%
  \href{http://www.ams.org/mathscinet-getitem?mr=#1}{#2}
}
\providecommand{\href}[2]{#2}

\end{document}